\documentclass[english]{article}

\usepackage{amssymb}
\usepackage{amsfonts}
\usepackage{graphicx}
\usepackage{amstext}
\usepackage{amsmath,amscd, amsthm}
\usepackage{euscript}
\usepackage[latin9]{inputenc}
\usepackage{geometry}
\usepackage{pdfsync}
\usepackage[all]{xy}
\usepackage{stmaryrd}
\usepackage{enumitem}
\usepackage{babel}
\usepackage{refstyle}
\usepackage{cite}
\usepackage{amssymb}
\usepackage{pict2e}

\usepackage[usenames,dvipsnames,svgnames,table]{xcolor}
\usepackage[driverfallback=hypertex, pdfborderstyle={},colorlinks=true]{hyperref}
\hypersetup{
     colorlinks   = true,
     linkcolor	  = DarkGreen,
     citecolor    = Purple}

\usepackage{cleveref}

\newlist{paraenum}{enumerate}{1}%
\setlist[paraenum]{label={(\arabic*)}}%
\newlist{lemenum}{enumerate}{1}%
\setlist[lemenum]{label={\emph{(\arabic*)}}}%

\makeatletter

\AtBeginDocument{\providecommand\thmref[1]{\ref{thm:#1}}}
\AtBeginDocument{\providecommand\propref[1]{\ref{prop:#1}}}
\AtBeginDocument{\providecommand\lemref[1]{\ref{lem:#1}}}
\AtBeginDocument{\providecommand\defref[1]{\ref{def:#1}}}
\AtBeginDocument{\providecommand\remref[1]{\ref{rem:#1}}}
\AtBeginDocument{\providecommand\conref[1]{\ref{con:#1}}}
\AtBeginDocument{\providecommand\corref[1]{\ref{cor:#1}}}

  \theoremstyle{plain}
  \newtheorem{thm}{Theorem}[subsection]

  \theoremstyle{definition}
  \newtheorem{con}[thm]{Construction}
  
  \theoremstyle{definition}
  \newtheorem{war}[thm]{Warning}

  \theoremstyle{plain}

  \theoremstyle{plain}
  \newtheorem{lem}[thm]{\protect\lemmaname}

  \theoremstyle{definition}
  \newtheorem{defn}[thm]{\protect\definitionname}

  \theoremstyle{plain}
  \newtheorem{prop}[thm]{\protect\propositionname}

  \theoremstyle{plain}
  \newtheorem{cor}[thm]{\protect\corollaryname}

  \theoremstyle{definition}

  \newtheorem{example}[thm]{\protect\examplename}

  \theoremstyle{remark}
  \newtheorem{rem}[thm]{\protect\remarkname}

  \theoremstyle{remark}
  \newtheorem*{rem*}{\protect\remarkname}

\@ifundefined{date}{}{\date{}}

\newref{prop}{name=Proposition~}
\newref{lem}{name=Lemma~}
\newref{thm}{name=Theorem~}
\newref{cor}{name=Corollary~}
\newref{rem}{name=Remark~}
\newref{def}{name=Definition~}
\newref{exa}{name=Example~}

\makeatother

  \providecommand{\corollaryname}{Corollary}
  \providecommand{\definitionname}{Definition}
  \providecommand{\examplename}{Example}
  \providecommand{\lemmaname}{Lemma}
  \providecommand{\propositionname}{Proposition}
  \providecommand{\remarkname}{Remark}

\newcommand*\noloc{%
        \nobreak
        \mskip6mu plus1mu
        \mathpunct{}%
        \nonscript
        \mkern-\thinmuskip
        {:}%
        \mskip2mu
        \relax
}

\begin{document}

\global\long\def\oto#1{\overset{#1}{\longrightarrow}}
\global\long\def\iso{\overset{\sim}{\longrightarrow}}
\global\long\def\into{\hookrightarrow}
\global\long\def\onto{\twoheadrightarrow}
\global\long\def\ss{\subseteq}
\global\long\def\adj{\leftrightarrows}
\global\long\def\bb#1{\mathbb{#1}}
\global\long\def\es{\varnothing}
\global\long\def\term{\text{pt}}
\global\long\def\cone{\triangleright}
\global\long\def\from{\leftarrow}

\global\long\def\colim{\operatorname*{colim}}
\global\long\def\lim{\operatorname*{lim}}
\global\long\def\map{\operatorname{Map}}
\global\long\def\End{\operatorname{End}}
\global\long\def\fun{\operatorname{Fun}}
\global\long\def\mul{\operatorname{Mul}}
\global\long\def\Id{\operatorname{Id}}
\global\long\def\red{\operatorname{\scriptsize{red}}}
\global\long\def\un{\operatorname{\scriptsize{un}}}
\global\long\def\final{\operatorname{\scriptsize{final}}}
\global\long\def\act{\operatorname{\scriptsize{act}}}
\global\long\def\grp{\operatorname{\scriptsize{grp}}}

\global\long\def\hocolim{\operatorname*{hocolim}}
\global\long\def\holim{\operatorname*{holim}}
\global\long\def\alg{\operatorname{Alg}}

\global\long\def\LP#1#2#3#4{\xymatrix{#1\ar[r]\ar[d]  &  #3\ar[d]\\
 #2\ar[r]\ar@{-->}[ru]  &  #4 
}
 }
\global\long\def\Square#1#2#3#4{\xymatrix{#1\ar[r]\ar[d]  &  #3\ar[d]\\
 #2\ar[r]  &  #4 
}
 }

\global\long\def\Rect#1#2#3#4#5#6{\xymatrix{#1\ar[r]\ar[d]  &  #3\ar[d]\ar[r]  &  #5\ar[d]\\
 #2\ar[r]  &  #4\ar[r]  &  #6 
}
 }

\global\long\def\catd{\mathbf{Cat}_{d}}
\global\long\def\cat{\mathbf{Cat}_{\infty}}
\global\long\def\catpt{\mathbf{Cat}_{\infty,*}}
\global\long\def\op{\mathbf{Op}_{\infty}}
\global\long\def\oppt{\mathbf{Op}_{\infty,*}}
\global\long\def\opred{\mathbf{Op}_{\infty}^{\red}}
\global\long\def\opun{\mathbf{Op}_{\infty}^{\un}}
\global\long\def\opunpt{\mathbf{Op}_{\infty,*}^{\un}}
\global\long\def\opredpt{\mathbf{Op}_{\infty,*}^{\red}}
\global\long\def\opd{\mathbf{Op}_{d}}
\global\long\def\pop{\mathbf{POp}_{\infty}}
\global\long\def\poppt{\mathbf{POp}_{\infty,*}}

\global\long\def\sseq{\mathbf{SSeq}}
\global\long\def\fin{\mathbf{Fin}}
\global\long\def\set{\mathbf{Set}}
\global\long\def\sset{\mathbf{sSet}}
\global\long\def\finpt{\mathbf{Fin}_{*}}
\global\long\def\uni{\mathbf{Uni}}
\global\long\def\com{\mathbf{Com}}
\global\long\def\ass{\mathbf{Ass}}
\global\long\def\triv{\mathbf{Triv}}

\title{The $\infty$-Categorical Eckmann\textendash Hilton Argument}

\author{Tomer M. Schlank\thanks{Einstein Institute of Mathematics, Hebrew University of Jerusalem.} \and Lior Yanovski\thanks{Einstein Institute of Mathematics, Hebrew University of Jerusalem.}}
\maketitle
\begin{abstract}
We define a reduced $\infty$-operad $\mathcal{P}$ to be $d$-connected
if the spaces $\mathcal{P}\left(n\right)$, of $n$-ary operations,
are $d$-connected for all $n\ge0$. Let $\mathcal{P}$ and $\mathcal{Q}$
be two reduced $\infty$-operads. We prove that if $\mathcal{P}$
is $d_{1}$-connected and $\mathcal{Q}$ is $d_{2}$-connected, then
their Boardman\textendash Vogt tensor product $\mathcal{P}\otimes\mathcal{Q}$
is $\left(d_{1}+d_{2}+2\right)$-connected. We consider this to be
a natural $\infty$-categorical generalization of the classical Eckmann\textendash Hilton
argument.
\end{abstract}

\tableofcontents

\section{Introduction}

\paragraph{Overview.}

The classical Eckmann\textendash Hilton argument (EHA), introduced in \cite{EH62},
states that given a set $X$ with two unital (ie having a two-sided
unit) binary operations 
\[
\circ,*\colon X\times X\to X, 
\]
if the two operations satisfy the ``interchange law''
\[
\forall a,b,c,d\in X,\quad\quad\left(a\circ b\right)*\left(c\circ d\right)=\left(a*c\right)\circ\left(b*d\right),
\]
then they coincide and, moreover, this unique operation is associative and commutative. Even though it is easy to prove, the EHA is very useful. The most familiar applications are the commutativity of the higher homotopy groups of a space and the commutativity of the fundamental group of an $H$-space.

A natural language for discussing different types of algebraic structures
and the interactions between them is that of operads (by which,
for now, we mean one-colored, symmetric operads in sets). For example,
the data of a unital binary operation on a set $X$ can be encoded
as a structure of an algebra on $X$ over a certain operad $\uni$.
Similarly, the data of a unital, associative, and commutative binary
operation on a set $X$ (namely, the structure of a commutative monoid)
can be encoded as an algebra structure on $X$ over the operad $\com$.
Furthermore, the category of operads is equipped with a tensor product
operation, introduced by Boardman and Vogt \cite{BV06},
such that given two operads $\mathcal{P}$ and $\mathcal{Q}$, a $\left(\mathcal{P}\otimes\mathcal{Q}\right)$-algebra
structure on a set $X$ is equivalent to a $\mathcal{P}$-algebra
structure and a $\mathcal{Q}$-algebra structure on $X$, which satisfy
a certain natural generalization of the interchange law defined above. Specializing
to the case at hand, one can rephrase the EHA as 
\[
\uni\otimes\uni\simeq\com.
\]

Noting that $\com$ is the terminal object in the category of operads
(as all operation sets are singletons), this formulation looks perhaps
a bit less surprising than the classical one. One can further observe
that we can replace $\uni$ by more general operads. We call an operad
$\mathcal{P}$ \emph{reduced} if both the set of nullary and the set
of unary operations of $\mathcal{P}$ are singletons (ie there is a unique constant and it serves
as a unit for all operations). The classical proof of the EHA can
be easily modified\footnote{Eg see Proposition 3.8 of \cite{FV15}.} to show that given two reduced operads $\mathcal{P}$
and $\mathcal{Q}$ whose $n$-ary operation sets are non-empty for
all $n$, we have 
\[
\mathcal{P}\otimes\mathcal{Q}\simeq\com.
\]
We call this the ``operadic formulation of the EHA''.

In many applications of the EHA, the two binary operations one starts
with are actually known to be associative in advance. This version,
which of course follows from the general EHA, can be stated as 
\[
\ass\otimes\ass\simeq\com,
\]
where $\ass$ is the operad that classifies the structure of a (unital,
associative) monoid. For future reference, we call this ``the associative
EHA''.

The language of operads already helps in organizing and systematizing
the study of ordinary algebraic structures, but it is really indispensable
for studying (and even defining) \emph{enriched} and \emph{homotopy
coherent} algebraic structures. To start with, by replacing the \emph{sets}
of $n$-ary operations of an operad with \emph{spaces} and requiring
the various composition and permutation maps to be continuous, one
obtains the notion of a topological operad. By further introducing
an appropriate notion of a weak equivalence, one can study homotopy
coherent algebraic structures. A fundamental example of such an object
is the little $n$-cubes topological operad $\bb E_{n}$ for $0\le n\le\infty$
(see, eg \cite{May72}). Loosely speaking, the structure of an $\bb E_{n}$-algebra
on a space $X$ can be thought of as a continuous unital multiplication
map on $X$ for which associativity holds up to a specified coherent
homotopy and commutativity also holds up to a specified coherent homotopy,
but only up to ``level $n$'' \footnote{The situation is slightly different for $n=0$ as an $\bb E_{0}$-algebra is just a pointed space.}. On a more technical level, $\bb E_{1}$ and $\bb E_{\infty}$ can
be interpreted as cofibrant models for $\ass$ and $\com$, respectively,
in a suitable model structure on the category of topological operads
(eg \cite{Vogt03}). The sequence $\bb E_{n}$ serves as
a kind of interpolation between them. 

There are many approaches to modeling ``homotopy coherent operads''
(both one-colored and multi-colored). Among them, the original
approach of J.P. May via specific topological operads \cite{May72}, via model structures (or partial versions thereof) on simplicial operads \cite{BM03, Vogt03, CM13a, Rob11} or dendroidal sets/spaces \cite{CM11,CM13b},  via ``operator categories'' of C. Barwick \cite{Bar18} or intrinsically to $\left(\infty,1\right)$-categories via analytic monads \cite{GHK17} or Day convolution \cite{Hau17}. 
We have chosen to work with the notion of $\infty$-operads
introduced and developed in J. Lurie's \cite{HA} based on the theory of $\infty$-categories
introduced by A. Joyal \cite{Joy05} and extensively developed in \cite{HTT}\footnote{See \cite{CHH16} for a discussion of the comparison of the different models.}. 
In this theory of $\infty$-operads (as in some of the others),
there is a notion analogous to the Boardman\textendash Vogt tensor product and
it is natural to ask whether there is also an analogue of the EHA.
For the associative EHA, one has the celebrated ``additivity theorem'',
proved by G. Dunn in the classical context \cite{Dun88} and by Lurie  in the language of $\infty$-operads \cite[Theorem 5.1.2.2]{HA}, which states that for all integers
$m,k\ge0$, we have
\[
\bb E_{m}\otimes\bb E_{k}\simeq\bb E_{m+k}.
\]
The goal of this paper is to state and prove an $\infty$-categorical
version of the classical (non-associative) EHA. The key observation
about the operadic formulation of the classical EHA is that both the
hypothesis regarding the non-emptiness of the operation sets of $\mathcal{P}$
and $\mathcal{Q}$ and the characterization of $\com$ as having singleton
operation sets can be phrased in terms of \emph{connectivity bounds}.
For an integer $d\ge-2$, we say that a reduced $\infty$-operad $\mathcal{P}$
is $d$-connected if all of its operation spaces are $d$-connected.
We prove
\begin{thm}
Given integers $d_{1},d_{2}\ge-2$ and two reduced $\infty$-operads
$\mathcal{P}$ and $\mathcal{Q}$, such that $\mathcal{P}$ is $d_{1}$-connected
and $\mathcal{Q}$ is $d_{2}$-connected, the $\infty$-operad $\mathcal{P}\otimes\mathcal{Q}$
is $\left(d_{1}+d_{2}+2\right)$-connected.
\end{thm}

Unlike in the classical case, our result does not imply the additivity
theorem (or vise versa), but the additivity theorem does demonstrate
the sharpness of our result, since $\bb E_{n}$ is $\left(n-2\right)$-connected
for all $n\ge0$. 

We shall deduce our $\infty$-categorical version of the EHA from
a ``relative'' version, which might be of independent interest.
For a reduced $\infty$-operad $\mathcal{P}$ and an integer $n\ge0$,
we denote by $\mathcal{P}\left(n\right)$ the space of $n$-ary operations
of $\mathcal{P}$. We say that a map of spaces is a $d$-equivalence,
if it induces a homotopy equivalence on $d$-truncations, and that
a map of reduced $\infty$-operads $\mathcal{P}\to\mathcal{Q}$ is
a $d$-equivalence, if for every integer $n\ge0,$ the map $\mathcal{P}\left(n\right)\to\mathcal{Q}\left(n\right)$
is a $d$-equivalence. 
\begin{thm}
\label{thm:The_Theorem_Introdction} Let $\mathcal{P}\to\mathcal{Q}$
be a $d$-equivalence of reduced $\infty$-operads and let $\mathcal{R}$
be a $k$-connected reduced $\infty$-operad. The map $\mathcal{P}\otimes\mathcal{R}\to\mathcal{Q}\otimes\mathcal{R}$
is a $\left(d+k+2\right)$-equivalence.
\end{thm}

This behavior of the Boardman\textendash Vogt tensor product on reduced $\infty$-operads
is somewhat analogous to the behavior of the join operation on spaces.
Given a map of spaces $X\to Y$ that is a $d$-equivalence and a
$k$-connected space $Z$, the map $X\star Z\to Y\star Z$ is a $\left(d+k+2\right)$-equivalence.
Incidentally, for the space of \emph{binary} operations we have $\left(\mathcal{P}\otimes\mathcal{R}\right)\left(2\right)\simeq\mathcal{P}\left(2\right)\star\mathcal{R}\left(2\right)$
(see \cite[Proposition 4.8]{FV15}), which relates the two phenomena. 

\paragraph{Outline of the proof.}

The proof of the classical EHA is straightforward. One simply uses
repeatedly the unitality and interchange law to deduce the various
equalities. For $\infty$-operads, the situation is considerably more
complicated as all identities hold only up to a specified coherent homotopy
and keeping track of this large amount of data is very difficult.
Consequently, there is probably no hope of writing down an explicit
formula for the operation spaces of $\mathcal{P}\otimes\mathcal{Q}$
in terms of those of $\mathcal{P}$ and $\mathcal{Q}$, except for
low degrees. Therefore, as usual with $\infty$-categories, one has
to adopt a less direct approach. 

The proof of \thmref{The_Theorem_Introdction} proceeds by a sequence
of reductions, which we now sketch in an informal way (we refer the
reader to the end of this section for a list of notational conventions).
An $\infty$-operad is called an essentially $d$-operad if all of
its multi-mapping spaces are homotopically $\left(d-1\right)$-truncated.
With every $\infty$-operad we can associate an essentially $d$-operad,
called its $d$-homotopy operad, by $\left(d-1\right)$-truncating
the multi-mapping spaces. This operation constitutes a left adjoint
to the inclusion of the full subcategory on essentially $d$-operads
into the $\infty$-category of $\infty$-operads. Using this adjunction
and the Yoneda lemma, a map of $\infty$-operads $f\colon \mathcal{P}\to\mathcal{Q}$
is a $d$-equivalence if and only if the induced map 
\[
\map\left(\mathcal{Q},\mathcal{R}\right)\to\map\left(\mathcal{P},\mathcal{R}\right)
\]
is a homotopy equivalence for every essentially $\left(d+1\right)$-operad
$\mathcal{R}$. Further analysis of the monad associated with
a reduced $\infty$-operad shows that when $\mathcal{P}$ and $\mathcal{Q}$
are reduced, it is enough to check the above equivalence only for
$\mathcal{R}$-s that are $\left(d+1\right)$-topoi endowed with
the Cartesian symmetric monoidal structure. 

Now let $\mathcal{P}\to\mathcal{Q}$ be a $d$-equivalence of reduced
$\infty$-operads and let $\mathcal{R}$ be a $k$-connected $\infty$-operad.
We want to show that the map $\mathcal{P}\otimes\mathcal{R}\to\mathcal{Q}\otimes\mathcal{R}$
is a $\left(d+k+2\right)$-equivalence. By the above reductions, it is enough
to show that for every $\left(d+k+3\right)$-topos $\mathcal{C}$
endowed with the Cartesian symmetric monoidal structure, the induced
map 
\[
\map\left(\mathcal{Q}\otimes\mathcal{R},\mathcal{C}\right)\to\map\left(\mathcal{P}\otimes\mathcal{R},\mathcal{C}\right)
\]
is a homotopy equivalence. A key property of the tensor product of
$\infty$-operads is that it endows the $\infty$-category of $\infty$-operads
with a symmetric monoidal structure that	 is \emph{closed}. Namely,
for every $\infty$-operad $\mathcal{O}$, there is an internal hom
functor $\alg_{\mathcal{O}}\left(-\right)$ that is right adjoint
to the tensor product $-\otimes\mathcal{O}$ \footnote{It is necessary to work here with multi-colored operads as developed in \cite{HA}, since even though the full subcategory of one-colored, or even reduced, 
$\infty$-operads is closed under the tensor product, the induced symmetric monoidal structure would not be closed.}. 
It is therefore enough to show that the map 
\[
\map\left(\mathcal{Q},\alg_{\mathcal{R}}\left(\mathcal{C}\right)\right)\to\map\left(\mathcal{P},\alg_{\mathcal{R}}\left(\mathcal{C}\right)\right)
\]
is a homotopy equivalence. Let $\triv$ be the trivial operad. There
are essentially unique maps $\triv\to\mathcal{P}$ and $\triv\to\mathcal{Q}$
that induce a commutative triangle
\[
\xymatrix{\map\left(\mathcal{Q},\alg_{\mathcal{R}}\left(\mathcal{C}\right)\right)\ar[dr]\ar[rr] &  & \map\left(\mathcal{P},\alg_{\mathcal{R}}\left(\mathcal{C}\right)\right)\ar[dl]\\
 & \map\left(\triv,\alg_{\mathcal{R}}\left(\mathcal{C}\right)\right)\simeq\alg_{\mathcal{R}}\left(\mathcal{C}\right)^{\simeq}
,}
\]
and it is enough to show that the top map induces an equivalence on
the fibers over each object $X$ of $\alg_{\mathcal{R}}\left(\mathcal{C}\right)$.
Fixing such an $X$, the fiber of the left map consists of the space
of ways to endow $X$ with the structure of a $\mathcal{Q}$-algebra.
Since $\mathcal{Q}$ is reduced, one can show that this is the space
of maps from $\mathcal{Q}$ to the so-called ``reduced endomorphism
operad of $X$''. This is a reduced $\infty$-operad $\End^{\red}\left(X\right)$
whose space of $n$-ary operations is roughly the space of maps $X^{n}\to X$
for which plugging the unique constant in all entries but one produces
the identity map of $X$. More formally, we have a homotopy fiber
sequence
\[
\End^{\red}\left(X\right)\left(n\right)\to\map\left(X^{n},X\right)\to\map\left(X^{\sqcup n},X\right)
\]
over the fold map $\nabla\colon X^{\sqcup n}\to X$. Consequently, by applying analogous reasoning to $\mathcal{P}$ and some naturality properties,
we are reduced to showing that for all $X$ in $\alg_{\mathcal{R}}\left(\mathcal{C}\right)$,
the induced map 
\[
\map\left(\mathcal{Q},\End^{\red}\left(X\right)\right)\to\map\left(\mathcal{P},\End^{\red}\left(X\right)\right)
\]
is a homotopy equivalence. Since $\mathcal{P}\to\mathcal{Q}$ is
a $d$-equivalence, it will suffice to show that $\End^{\red}\left(X\right)$
is an essentially $\left(d+1\right)$-operad. Namely, we need only to show that the spaces
$\End^{\red}\left(X\right)\left(n\right)$ are $d$-truncated. Using
the homotopy fiber sequence above, we may present $\End^{\red}\left(X\right)\left(n\right)$
as the space of lifts in the commutative square
\[
\xymatrix{X^{\sqcup n}\ar[d]\ar[r]^{\nabla} & X\ar[d]\\
X^{n}\ar[r]\ar@{-->}[ru] & \term
.}
\]

The underlying $\infty$-category of $\alg_{\mathcal{R}}\left(\mathcal{C}\right)$
is an essentially $\left(d+k+3\right)$-category (since $\mathcal{C}$
is); hence the right vertical map is $\left(d+k+2\right)$-truncated.
We show that in a general presentable $\infty$-category, the space
of lifts of an $n$-connected map against an $m$-truncated map is
$\left(m-n-2\right)$-truncated. It is therefore enough to show that
the map $X^{\sqcup n}\to X^{n}$ is $k$-connected in $\alg_{\mathcal{R}}\left(\mathcal{C}\right)$.
Under suitable conditions, which are satisfied in our situation, the
$k$-connectedness of a map of algebras over an $\infty$-operad can be detected
on the level of the underlying objects. Using the fact that $\mathcal{C}$
is an $\infty$-\emph{topos} we are reduced to proving that the map $X^{\sqcup n}\to X^{n}$
has a section and that it becomes an equivalence after $k$-truncation
in $\mathcal{C}$. For the first assertion, we show that one can construct
a section rather easily using any $n$-ary operation of $\mathcal{R}$
for $n\ge2$. The second assertion follows from the fact that $\mathcal{R}$
itself is $k$-connected, and so, roughly speaking, after $k$-truncation
we can replace $\mathcal{R}$ with $\bb E_{\infty}$ and the coproduct
of $\bb E_{\infty}$-algebras coincides with the product. The $\infty$-categorical
EHA now follows easily from this by taking $\mathcal{Q}=\bb E_{\infty}$.

\paragraph{Organization.}

The paper is organized as follows. In Section 2, we develop some general
theory regarding reduced (and unital) $\infty$-operads. The first
theme is the construction and analysis of the reduced endomorphism
operad. The second is an explicit formula for the associated map of
monads induced from a map of reduced $\infty$-operads. 

In Section 3, we recall from \cite{SY19} some basic definitions and properties of essentially $d$-categories (and operads), as well as the notion of a $d$-homotopy category (and operad). We then proceed to prove that a map of $\infty$-operads is a $d$-equivalence if and only if it induces an equivalence on the spaces of algebras in every $\left(d+1\right)$-topos endowed with the Cartesian symmetric monoidal structure. 

In Section 4, we prove some general results regarding the notions
of $d$-connected and $d$-truncated morphisms in presentable $\infty$-categories.

In Section 5 we prove the main results of the paper. In particular
we prove \thmref{The_Theorem_Introdction} and the $\infty$-categorical
Eckmann\textendash Hilton argument as a corollary. We also include a couple of
simple applications. 

For a more detailed outline we refer the reader to the individual
introduction of each section. 

Much of the length of the paper is due to the careful and detailed
verification of many lemmas in $\infty$-category theory, whose proofs
are arguably straightforward, but nonetheless do not appear in the
literature. This refers mainly to the material up to subsection 4.3,
from which the main theorems are \propref{Operad_Reduction_Right_Adjoint},
\propref{d-homotopy_operad}, and \propref{d-equivalence}. Having said
that, we believe that the theory and language of $\infty$-categories
in general and $\infty$-operads in particular is still in an early
enough stage of development to justify full detailed proofs of every
claim that has no reference (known to the authors) in the literature.
Hopefully, the added value in terms of rigor and accessibility to
non-experts compensates for the loss in brevity and elegance of exposition. 

\paragraph{Acknowledgments.}
We would like to thank Julie Bergner and Jim Stasheff, as well as all the participants of the Seminarak group, for useful discussions about the subject of this paper. We also thank the referee for helpful comments and corrections. The first Author was supported by the Alon Fellowship and the ISF grant 1588/18 and the second author was supported by the ISF grant 1650/15.

\paragraph{Conventions.}

We work in the setting of $\infty$-categories (a.k.a. quasi-categories)
and $\infty$-operads, relying heavily on the results of \cite{HTT}
and \cite{HA}. Since we have numerous references to these two foundational
works, references to \cite{HTT} are abbreviated
as T.? and those to \cite{HA} as A.? while other references
are cited in the standard way. As a rule, we follow the notation
of \cite{HTT} and \cite{HA} whenever possible. However, we supplement this notation and deviate from it in several cases in which we believe this enhances readability. In particular:
\begin{enumerate}

\item We abuse notation by identifying an ordinary category $\mathcal{C}$
with its nerve $N\left(\mathcal{C}\right)$. 

\item We use the symbol $\term_{\mathcal{C}}$ to denote the terminal object
of an $\infty$-category $\mathcal{C}$ (or just $\term$ if $\mathcal{C}$
is clear from the context). 

\item We abbreviate the data of an $\infty$-operad $p\colon \mathcal{O}^{\otimes}\to\finpt$
by $\mathcal{O}$ and reserve the notation $\mathcal{O}^{\otimes}$
for the $\infty$-category that is the source of $p$. Similarly,
given two $\infty$-operads $\mathcal{O}$ and $\mathcal{U}$, we 
write $f\colon \mathcal{O}\to\mathcal{U}$ for a map of $\infty$-operads
from $\mathcal{O}$ to $\mathcal{U}$. The underlying $\infty$-category
of $\mathcal{O}$, which in \cite{HA} is denoted by $\mathcal{O}_{\left\langle 1\right\rangle }^{\otimes}$,
is here denoted by $\underline{\mathcal{O}}$. 

\item When the $\infty$-operad is a symmetric monoidal $\infty$-category,
we usually denote it by $\mathcal{C}$ or $\mathcal{D}$. We
will sometimes abuse notation and write $\mathcal{C}$ also for the
underlying $\infty$-category $\underline{\mathcal{C}}$ when there
is no chance of confusion.

\item By a \emph{presentably symmetric monoidal $\infty$-category}, we
mean a symmetric monoidal $\infty$-category $\mathcal{C}$, such
that the underlying $\infty$-category $\underline{\mathcal{C}}$
is a presentable $\infty$-category and the tensor product preserves
colimits separately in each variable. 

\item Given two $\infty$-operads $\mathcal{O}$ and $\mathcal{U}$, we
denote by $\alg_{\mathcal{O}}\left(\mathcal{U}\right)$ the $\infty$-operad
$\alg_{\mathcal{O}}\left(\mathcal{U}\right)^{\otimes}\to\finpt$ from
Example A.3.2.4.4. This is the internal mapping object induced
from the closed symmetric monoidal structure on $\op$ (see A.2.2.5.13).
The underlying $\infty$-category $\underline{\alg}_{\mathcal{O}}\left(\mathcal{U}\right)$
is the usual $\infty$-category of $\mathcal{O}$-algebras in $\mathcal{U}$
(which in \cite{HA} is denoted by $\alg_{\mathcal{O}}\left(\mathcal{U}\right)$).
Moreover, the maximal Kan sub-complex $\underline{\alg}_{\mathcal{O}}\left(\mathcal{U}\right)^{\simeq}$
is the space of morphisms $\map_{\op}\left(\mathcal{O},\mathcal{U}\right)$
from $\mathcal{O}$ to $\mathcal{U}$ as objects of the $\infty$-category
$\op$. Recall from A.3.2.4.4 that for a symmetric monoidal $\infty$-category
$\mathcal{C}$, the $\infty$-operad $\alg_{\mathcal{O}}\left(\mathcal{C}\right)$
is again symmetric monoidal and for every object $X\in\underline{\mathcal{O}}$,
the evaluation functors $ev_{X}\colon \alg_{\mathcal{O}}\left(\mathcal{C}\right)\to\mathcal{C}$
are symmetric monoidal functors.

\item Let $\mathcal{C}$ be an $\infty$-category. We denote the corresponding
coCartesian $\infty$-operad $\mathcal{C}^{\sqcup}\to\finpt$ by $\mathcal{C}_{\sqcup}$
(see Definition A.2.4.3.7). If $\mathcal{C}$ has all finite products,
we denote the Cartesian symmetric monoidal $\infty$-category \emph{$\mathcal{C}^{\times}\to\finpt$}
by $\mathcal{C}_{\times}$ (see Construction A.2.4.1.4).
\end{enumerate}

\section{Reduced $\infty$-Operads}

In this section we develop some general theory of unital and
reduced $\infty$-operads. In 2.1 we establish some formal results
for adjunctions and under categories. In 2.2 we specialize the
results of 2.1 to prove that the inclusion of reduced $\infty$-operads
into pointed unital $\infty$-operads admits a right adjoint, and
analyze it. More precisely, given a unital $\infty$-operad $\mathcal{O}$
and an object $X\in\underline{\mathcal{O}}$ we define a reduced $\infty$-operad
$\End_{\mathcal{O}}^{\red}\left(X\right)$, which we call the reduced
endomorphism operad of $X$, and show that it satisfies a universal
property. Moreover, we give an explicit description of $\End_{\mathcal{C}}^{\red}\left(X\right)$,
which will be fundamental in analyzing the truncatedness of its spaces
of operations.

In 2.3 we discuss the underlying symmetric sequence of a reduced $\infty$-operad
and in 2.4 we use it to write an explicit formula for the free algebra
over an $\infty$-operad (this is essentially a reformulation of A.3.1.3).
The material of the last two subsections is well known in the 1-categorical
setting and will come as no surprise to anyone familiar with the subject.
We note that in \cite{Hau17}, Haugseng develops a theory
of $\infty$-operads using this approach and compares it with other
models including Lurie's $\infty$-operads, though as far as we know,
the precise results for algebras have not been furnished yet.
Thus, we take it upon ourselves to flesh out the details of the little
part of this theory that is required for our purposes. 

\subsection{Adjunctions and Under-categories}

We begin with some formal general observations on adjunctions
and under-categories. 

\begin{lem}
\label{lem:Adjoint_Slice_Category} Let $R\colon \mathcal{D}\adj\mathcal{C}\noloc L$
be an adjunction between $\infty$-categories. For every object $X\in\mathcal{C}$
there is a canonical equivalence of $\infty$-categories $\mathcal{D}_{L\left(X\right)/}\simeq\mathcal{D}\times_{\mathcal{C}}\mathcal{C}_{X/}$. 
\end{lem}

\begin{proof}
We denote the $\infty$-category $\mathcal{D}\times_{\mathcal{C}}\mathcal{C}_{X/}$
by $\mathcal{D}_{X/}$. Let $\eta\colon X\to RL\left(X\right)$ be the $X$-component
of the unit of the adjunction $L\dashv R$. By T.2.1.2.1, the projections
$p_{0}\colon \mathcal{C}_{\eta/}\to\mathcal{C}_{X/}$ and $p_{1}\colon \mathcal{C}_{\eta/}\to\mathcal{C}_{RL\left(X\right)/}$
are left fibrations. Moreover, since $\Delta^{\left\{ 1\right\} }\into\Delta^{1}$
is right anodyne, the map $p_{1}$ is an equivalence of $\infty$-categories.
By T.2.2.3.3, we can choose an inverse $p_{1}^{-1}\colon \mathcal{C}_{RL\left(X\right)/}\to\mathcal{C}_{\eta/}$
to $p_{1}$ that strictly commutes with the projections to $\mathcal{C}$.
We obtain a commutative diagram of simplicial sets
\[
\xymatrix{\mathcal{D}_{L\left(X\right)/}\ar[d]\ar[r] & \mathcal{C}_{RL\left(X\right)/}\ar[d]\ar[r]^-{p_{0}p_{1}^{-1}} & \mathcal{C}_{X/}\ar[d]\\
\mathcal{D}\ar[r] & \mathcal{C}\ar@{=}\ar[r] & \mathcal{C}
.}
\]
There is an induced map from the upper left corner to the pullback of the outer rectangle without the upper left corner, which is another commutative diagram of simplicial sets

\[
\xymatrix{\mathcal{D}_{L\left(X\right)/}\ar[rr]\ar[rd] & \mathcal{} & \mathcal{D}_{X/}\ar[ld]\\
 & \mathcal{D}
.}
\]
Since left fibrations are closed under base change (T.2.1.2.1), the
vertical maps are left fibrations over $\mathcal{D}$. Hence, to show
that the top map is an equivalence it is enough to show that the induced
map on fibers is a homotopy equivalence (T.2.2.3.3). For every $Y\in\mathcal{D}$
we get a map 
\[
\map_{\mathcal{D}}^{R}\left(L\left(X\right),Y\right)\to\map_{\mathcal{C}}^{R}\left(X,R\left(Y\right)\right),
\]
which is by construction obtained by applying the functor $R$ and
pre-composing with the unit $\eta\colon X\to RL\left(X\right)$. By the
universal property of the unit map this is a homotopy equivalence
for all $Y\in\mathcal{D}$ and therefore the map $\mathcal{D}_{L\left(X\right)/}\to\mathcal{D}_{X/}$
is an equivalence of $\infty$-categories.
\end{proof}

\begin{lem}
\label{lem:Local_Right_Adjoint}Let $L\colon \mathcal{C}\adj\mathcal{D} \noloc R$
be an adjunction of $\infty$-categories and let $X\in\mathcal{C}$. The
induced functor 
\[
L_{X}\colon \mathcal{C}_{X/}\to\mathcal{D}_{L\left(X\right)/}
\]
has a right adjoint $R_{X}$. Moreover, if $R$ is fully faithful,
then $R_{X}$ is also fully faithful.
\end{lem}

\begin{proof}
Let $p\colon \mathcal{M}\to\Delta^{1}$ be the coCartesian fibration associated with the functor $L$ (which is also Cartesian, since $L$ has a right adjoint). 
We can assume that we have a commutative diagram
\[
\xymatrix{\Delta^{1}\times\mathcal{C}\ar[rr]^-{s}\ar[dr] &  & \mathcal{M}\ar[dl]^-{p}\\
 & \Delta^{1}
,}
\]
such that $s|_{\Delta^{\left\{ 0\right\} }\times\mathcal{C}}=\Id$,
$s|_{\Delta^{\left\{ 1\right\} }\times\mathcal{C}}=L$ and $s|_{\Delta^{1}\times\left\{ X\right\} }$
is a coCartesian edge of $\mathcal{M}$ for every $X\in\mathcal{C}$
(combine T.5.2.1.1 and T.5.2.1.3). It is clear from T.1.2.9.2 that
for any pair of $\infty$-categories with objects $X\in\mathcal{C}$
and $Y\in\mathcal{D}$ there is a canonical isomorphism 
\[
\left(\mathcal{C}\times\mathcal{D}\right)_{\left(X,Y\right)/}\simeq\mathcal{C}_{X/}\times\mathcal{D}_{Y/}.
\]

Hence, we get an induced commutative diagram
\[
\xymatrix{\Delta^{1}\times\mathcal{C}_{X/}\simeq\left(\Delta^{1}\times\mathcal{C}\right)_{\left(0,X\right)/}\ar[rr]^-{s}\ar[dr] &  & \mathcal{M}_{X/}\ar[dl]^-{p_{X}}\\
 & \Delta^{1}\simeq\Delta_{0/}^{1}
.}
\]
The functor $p_{X}$ is a Cartesian and coCartesian fibration by the
duals of T.2.4.3.1(1) and T.2.4.3.2(1). Moreover, an edge in $\mathcal{M}_{X/}$
is (co)Cartesian if and only if its projection to $\mathcal{M}$ is
(co)Cartesian by the duals of T.2.4.3.1(2) and T.2.4.3.2(2), which
shows that the functor $L_{X}$ is associated with $p_{X}$. It follows
that $L_{X}$ has a right adjoint $R_{X}$. 

Assuming that $R$ is fully faithful, we will show that $R_{X}$ is
fully faithful by showing that the counit of the adjunction $L_{X}\dashv R_{X}$
is an equivalence. For every object, the counit map is an edge of
$\mathcal{M}_{X/}$. Since the projection $\mathcal{M}_{X/}\to\mathcal{M}$
is conservative, it is enough to show that the counit map of $L_{X}\dashv R_{X}$
is mapped to the counit map of $L\dashv R$. Indeed, for an object
$Y\in\mathcal{D}\simeq\mathcal{M}|_{\Delta^{\left\{ 1\right\} }}$,
we choose a Cartesian edge $e\colon R\left(Y\right)\to Y$ and a coCartesian
edge $d\colon R\left(Y\right)\to L\left(R\left(Y\right)\right)$, and combine
them into a commutative diagram of the form:

\[
\xymatrix{\Lambda_{0}^{2}\ar[r]^{f}\ar[d] & \mathcal{M}\ar[d]^{p}\\
\Delta^{2}\ar[r] & \Delta^{1}
,}
\]
where $f|_{\Delta^{\left\{ 0,1\right\} }}=d$ and $f|_{\Delta^{\left\{ 0,2\right\} }}=e$.
Since $d$ is coCartesian, there exists a lift 
$\overline{f}\colon \Delta^{2}\to\mathcal{M}$ 
that gives an edge 
\[
\overline{f}|_{\Delta^{\left\{ 1,2\right\} }}=c\colon L\left(R\left(Y\right)\right)\to Y
\]
that is isomorphic to the counit map of the adjunction $L\dashv R$
at $Y$ in the homotopy category $h\mathcal{D}$. We can similarly construct
the counit map for an object of $\mathcal{M}_{X/}$. The assertion
now follows from the above characterization of (co)Cartesian edges
in $\mathcal{M}_{X/}$.
\end{proof}
\begin{lem}
Let $F\colon \mathcal{C}\to\mathcal{D}$ be a functor that preserves pullbacks; then $F_{X}\colon \mathcal{C}_{X/}\to\mathcal{D}_{F\left(X\right)/}$ also
preserves pullbacks.
\end{lem}

\begin{proof}
Consider the commutative square
\[
\xymatrix{\mathcal{C}_{X/}\ar[r]\ar[d] & \mathcal{D}_{F\left(X\right)/}\ar[d]\\
\mathcal{C}\ar[r] & \mathcal{D}
.}
\]
The vertical functors and the bottom horizontal functor preserve pullbacks.
The right vertical functor is conservative. It follows that the top
horizontal functor preserves pullbacks as well. 
\end{proof}
\begin{defn}
\label{def:reduced_object}Let $F\colon \mathcal{C}\to\mathcal{D}$ be a
functor between $\infty$-categories. We say that an object $Y$ is \emph{reduced} if $F\left(Y\right)$ is initial in $\mathcal{D}$. We define $\mathcal{C}^{\red}$ to be the full subcategory of $\mathcal{C}$ spanned by the reduced objects ($F$ will always be clear from the context when we employ this terminology).
\end{defn}

\begin{prop}
\label{prop:Reduction_co_localization}Let $L\colon \mathcal{C}\adj\mathcal{D}\noloc R$
be an adjunction between $\infty$-categories. Assume that $\mathcal{C}$
admits and $L$ preserves pullbacks, that $\mathcal{D}$ admits an initial object, and that $R$ is fully faithful. For every object $Y\in\mathcal{C}$ we
consider the following pullback diagram
\[
\xymatrix{Y^{\red}\ar[d]\ar[r] & Y\ar[d]\\
R\left(\es_{\mathcal{D}}\right)\ar[r] & RL\left(Y\right)
,}
\]
where the right vertical map is the unit map of $Y$ and the bottom
horizontal map is the image under $R$ of the essentially unique map
$\es_{\mathcal{D}}\to L\left(Y\right)$. The top horizontal map $\rho\colon Y^{\red}\to Y$
exhibits $Y^{\red}$ as a co-localization of $Y$ with respect to $\mathcal{C}^{\red}$
(dual to T.5.2.7.6).
\end{prop}

\begin{proof}
First, we show that $Y^{\red}$ is in fact reduced. Applying $L$ to
the defining diagram of $Y^{\red}$ and using the fact that $L$ preserves
pullbacks, we see that the map $L\left(Y^{\red}\right)\to LR\left(\es\right)$
is the pullback of the map $L\left(Y\right)\to LRL\left(Y\right)$,
which is an equivalence (from the fact that the counit $LR\left(Y\right)\to Y$ is an equivalence, the zig-zag identities and the 2-out-of-3 property). It follows that the map $L\left(Y^{\red}\right)\to LR\left(\es\right)$ is an equivalence,
but $LR\left(\es\right)\to\es$ is an equivalence as well (since $R$
is fully faithful) and we are done. 

Now, we show that $\rho$ is a co-localization. Let $Z$ be a reduced
object. We have a homotopy pullback diagram of spaces
\[
\xymatrix{\map\left(Z,Y^{\red}\right)\ar[d]\ar[r] & \map\left(Z,Y\right)\ar[d]\\
\map\left(Z,R\left(\es\right)\right)\ar[r] & \map\left(Z,RL\left(Y\right)\right)
}
\]
and we note that the space of maps from a reduced object to any object
in the essential image of $R$ is contractible. 
\end{proof}
\begin{cor}
\label{cor:Reduction_Adjoint_General}In the setting of \propref{Reduction_co_localization},
the inclusion $\mathcal{C}^{\red}\into\mathcal{C}$ admits a right
adjoint and the co-localization map $Y^{\red}\to Y$ can be taken to
be the counit of the adjunction at $Y$.
\end{cor}

\begin{proof}
Follows from \propref{Reduction_co_localization} and the dual of
T.5.2.7.8.
\end{proof}

\subsection{Pointed Unital and Reduced $\infty$-operads}

Recall from \cite{HA} the following definitions:
\begin{defn}
\label{def:Unital_Reduced_Operad}(A.2.3.1.1, A.2.3.4.1) 
An $\infty$-operad \emph{$\mathcal{O}$ }is called:
\begin{paraenum}
\item \emph{Unital} if for every object $X$ of $\underline{\mathcal{O}}$,
the space of constants $\text{Mul}_{\mathcal{O}}\left(\es,X\right)$
is contractible. We denote the full $\infty$-category spanned by
the unital $\infty$-operads by $\opun$. 
\item \emph{Reduced }if it is\emph{ }unital and the underlying $\infty$-category
is a contractible space. We denote the full $\infty$-category spanned
by the reduced $\infty$-operads by $\opred$.
\end{paraenum}
\end{defn}

\begin{example}
A symmetric monoidal $\infty$-category is unital if and only if the
unit object is initial.
\end{example}

We proceed by listing the various adjunctions between the different
$\infty$-categories of $\infty$-operads and $\infty$-categories.
First, recall from A.2.1.4.10 that there is an underlying $\infty$-category
functor $\underline{\left(-\right)}\colon \op\to\cat$ and that this functor
has a left adjoint $\iota\colon \cat\into\op$, which is a fully faithful embedding.
Informally, $\iota$ regards an $\infty$-category as an $\infty$-operad
with empty higher (and nullary) multi-mapping spaces. On the other
hand,
\begin{lem}
\label{lem:coCart_Left_Adjoint}The restriction of the forgetful functor
$\underline{\left(-\right)}\colon \opun\to\cat$ admits a right adjoint
that takes every $\infty$-category $\mathcal{C}$ to the coCartesian
$\infty$-operad $\mathcal{C}_{\sqcup}$ and the unit map of the adjunction
$\underline{\mathcal{C}_{\sqcup}}\to\mathcal{C}$ is an equivalence
(ie $\left(-\right)_{\sqcup}$ is fully faithful).
\end{lem}

\begin{proof}
The first claim follows from A.2.4.3.9 by passing to maximal $\infty$-subgroupoids.
The second claim follows from A.2.4.3.11.
\end{proof}
By A.2.3.1.9, the fully faithful embedding $\opun\into\op$ has a left adjoint given by tensoring with $\bb E_{0}$ (which is a localization
functor). From this follows,
\begin{lem}
\label{lem:E_0_initial}$\bb E_{0}$ is the initial object of $\opred$.
\end{lem}

\begin{proof}
The composition of forgetful functors 
\[
\opun\to\op\to\cat
\]
has a left adjoint given as the composition of the corresponding left
adjoints. The first one takes $\Delta^{0}$ to $\mathbf{Triv}$ (by
A.2.1.4.8) and the second takes $\mathbf{Triv}$ to $\mathbf{Triv}\otimes\bb E_{0}\simeq\bb E_{0}$
(by A.2.3.1.9). Hence, for every reduced operad $\mathcal{P}$ (which
is in particular unital), we get 
\[
\map\left(\bb E_{0},\mathcal{P}\right)\simeq\map\left(\Delta^{0},\underline{\mathcal{P}}\right)\simeq\underline{\mathcal{P}}^{\simeq}\simeq\Delta^{0}.
\]
\end{proof}
One source of unital symmetric monoidal $\infty$-categories is
\begin{lem}
\label{lem:alg_unital} Let $\mathcal{Q}$ be a unital $\infty$-operad
and let $\mathcal{C}$ be a symmetric monoidal $\infty$-category. The symmetric monoidal $\infty$-category $\alg_{\mathcal{Q}}\left(\mathcal{C}\right)$
is also unital.
\end{lem}

\begin{proof}
By A.3.2.4.4, the $\infty$-operad $\alg_{\mathcal{Q}}\left(\mathcal{C}\right)^{\otimes}\to\finpt$
is also a symmetric monoidal $\infty$-category and so we only need to
show that the unit object of $\alg_{\mathcal{Q}}\left(\mathcal{C}\right)$
is initial. Since $\mathcal{Q}$ is unital, the canonical map $\mathcal{Q}\to\bb E_{0}\otimes\mathcal{Q}$
is an equivalence of $\infty$-operads (by A.2.3.1.9) and therefore
the forgetful functor
\[
\underline{\alg}_{\bb E_{0}\otimes\mathcal{Q}}\left(\mathcal{C}\right)\simeq\underline{\alg}_{\bb E_{0}}\left(\alg_{\mathcal{Q}}\left(\mathcal{C}\right)\right)\to\underline{\alg}_{\mathcal{Q}}\left(\mathcal{C}\right)
\]
is an equivalence of $\infty$-categories. On the other hand, by A.2.1.3.10
we have 
\[
\underline{\alg}_{\bb E_{0}}\left(\alg_{\mathcal{Q}}\left(\mathcal{C}\right)\right)\simeq\underline{\alg}_{\mathcal{Q}}\left(\mathcal{C}\right)_{1/},
\]
where $1\in\underline{\alg}_{\mathcal{Q}}\left(\mathcal{C}\right)$
is the unit object and the projection $\underline{\alg}_{\mathcal{Q}}\left(\mathcal{C}\right)_{1/}\to\underline{\alg}_{\mathcal{Q}}\left(\mathcal{C}\right)$
is an equivalence of $\infty$-categories if and only if $1$ is initial
(T.1.2.12.5).
\end{proof}
\begin{defn}
\label{def:Pointed_Operads} The $\infty$-category of pointed $\infty$-categories
is denoted by $\catpt=\left(\cat\right)_{\Delta^{0}/}$. The $\infty$-category
of pointed $\infty$-operads is denoted by $\oppt=\op\times_{\cat}\catpt$.
We also denote by $\opunpt$ and $\opredpt$ the corresponding $\infty$-categories
of pointed unital (resp. reduced) $\infty$-operads. 
\end{defn}

\begin{rem}
\label{rem:Pointed_Operads} Using \lemref{Adjoint_Slice_Category},
we have an equivalence of $\infty$-categories $\oppt\simeq\left(\op\right)_{\triv/}$.
Since $\triv\to\bb E_{0}$ is an equivalence after tensoring with
$\bb E_{0}$, we get $\opunpt\simeq\left(\opun\right)_{\bb E_{0}/}$
and therefore also $\opredpt\simeq(\opred)_{\bb E_{0}/}$.
We allow ourselves to pass freely between the two points of view on
(unital, reduced) pointed $\infty$-operads.
\end{rem}

\begin{rem}
Observe that by \lemref{E_0_initial}, the projection $\opredpt\to\opred$
is an equivalence. Hence, the inclusion $\opred\into\opun$ induces
a functor
\[
\opred\simeq\opredpt\into\opunpt.
\]
Moreover, it exhibits $\opred$ as the full subcategory of $\opunpt$
spanned by the reduced objects with respect to the underlying $\infty$-category
functor $\underline{\left(-\right)}\colon \opunpt\to\cat$ in the sense
of \defref{reduced_object}. Thus, the two notions of
``reduced $\infty$-operad'' coincide.
\end{rem}

We now apply the general observations from the previous subsection
to deduce the following:
\begin{prop}
\label{prop:Operad_Reduction_Right_Adjoint}The inclusion 
\[
\opred\simeq\opredpt\into\opunpt
\]
has a right adjoint $\left(-\right)^{\red}$. Moreover, for a pointed
unital $\infty$-operad $\mathcal{Q}$ the value of the right adjoint
is given by the pullback 
\[
\xymatrix{\mathcal{Q}^{\red}\ar[r]\ar[d] & \mathcal{Q}\ar[d]\\
\bb E_{\infty}\ar[r] & \underline{\mathcal{Q}}_{\sqcup}
}
\]
in the $\infty$-category $\opunpt$. Furthermore, the top map can be
taken to be the counit of the adjunction at $\mathcal{Q}$. 
\end{prop}

\begin{proof}
We need to verify the hypothesis of \propref{Reduction_co_localization}.
The underlying $\infty$-category functor $L\colon \opun\to\cat$ is a composition of two functors $\opun\to\op\to\cat$. 
The first is a right adjoint by A.2.3.1.9 and the second is a right adjoint by A.2.1.4.10. Hence, the composition is a right adjoint as well and therefore preserves limits. 
Moreover, by \lemref{coCart_Left_Adjoint}, $L$ is also a left adjoint and its right adjoint is fully faithful. 
Hence, the functor $\opunpt\to\catpt$ also has a fully faithful right adjoint and we have $\left(\opunpt\right)^{\red}\simeq\opredpt$. 
Finally, by
\propref{Reduction_co_localization}, the inclusion $\opred\simeq\opredpt\into\opunpt$ admits a right adjoint with the stated description.
\end{proof}

\begin{defn}
\label{def:Reduced_Endomorphism_Operad} A unital $\infty$-operad
$\mathcal{Q}$ and an object $X\in\underline{\mathcal{Q}}$ determine
a pointed unital $\infty$-operad $\mathcal{Q}_{X}\in\opunpt$. We
denote $\End_{\mathcal{Q}}^{\red}\left(X\right):=\left(\mathcal{Q}_{X}\right)^{\red}$
and call it the \emph{reduced endomorphism $\infty$-operad} of $X$
in $\mathcal{Q}$. 
\end{defn}

We can describe the reduced $\infty$-operad $\End_{\mathcal{Q}}^{\red}\left(X\right)$
informally as follows. For every $m\in\bb N$, denote by $X^{\left(m\right)}$ the $m$-tuple $\left(X,\dots,X\right)$.
The space of $m$-ary operations is the ``subspace'' of $\text{Mul}_{\mathcal{Q}}\left(X^{\left(m\right)},X\right)$
of those maps that are reduced in the sense that plugging the unique
constant in all arguments but one results in an identity morphism
$X\to X$. We end this subsection by making the above description
precise in a special case of a symmetric monoidal $\infty$-category.
For this, we first need to analyze the way multi-mapping spaces interact with limits of $\infty$-operads. 

For every integer $m$, there is a functor $hG^{(m)}\colon h\oppt\to h\mathcal{S}$, that takes each $\infty$-operad $\mathcal{P}$
pointed by an object $X$ to the space $\mul_{\mathcal{P}}\left(X^{(m)},X\right)$ and a map of pointed $\infty$-operads $f\colon \mathcal{P}\to\mathcal{Q}$ to the homotopy class of the induced map on multi-mapping spaces 
$\mul_{\mathcal{P}}(X^{(m)},X)\to\mul_{\mathcal{Q}}(f(X)^{(m)},f(X))$.

\begin{lem}
\label{lem:Limits_Operads} For every integer $m$, there is a limit-preserving functor $G^{(m)}\colon \oppt\to\mathcal{S}$, that lifts the functor $hG^{(m)}\colon h\oppt\to h\mathcal{S}$.
\end{lem}

\begin{proof}
Recall the  combinatorial simplicial model category $\pop$ of $\infty$-preoperads, whose underlying $\infty$-category is $\op$ (see A.2.1.4). 
Let $\overline{\mathcal{Z}}_{0}\ss\overline{\mathcal{Z}}_{1}\ss\fin_{*}$
be the following subcategories: 
\begin{paraenum}
\item The category $\overline{\mathcal{Z}}_{0}$ is discrete and contains
only the objects $\left\langle 1\right\rangle $ and $\left\langle m\right\rangle $. 
\item The category $\overline{\mathcal{Z}}_{1}$ contains $\overline{\mathcal{Z}}_{0}$
together with a unique non-identity morphism, which is the active map
$\alpha\colon \left\langle m\right\rangle \to\left\langle 1\right\rangle $.
\end{paraenum}
We endow $\overline{\mathcal{Z}}_{0}$ and $\overline{\mathcal{Z}}_{1}$
with the induced (trivial) marking. Unwinding the definition, for
any $\infty$-operad $\mathcal{P}$, the simplicial set $\map_{\pop}\left(\overline{\mathcal{Z}}_{0},\mathcal{P}^{\natural}\right)$
is \emph{isomorphic} to $\mathcal{P}_{\left\langle m\right\rangle }^{\simeq}\times\mathcal{P}_{\left\langle 1\right\rangle }^{\simeq}$.
Moreover, given 
\[
\underline{X}=\left(X_{1}\oplus\cdots\oplus X_{m},Y\right)\in\mathcal{P}_{\left\langle m\right\rangle }^{\simeq}\times\mathcal{P}_{\left\langle 1\right\rangle }^{\simeq},
\]
the fiber of the fibration (hence also the \emph{homotopy fiber})
\[
\varphi_{\mathcal{P}}\colon \map_{\pop}\left(\overline{\mathcal{Z}}_{1},\mathcal{P}^{\natural}\right)\to\map_{\pop}\left(\overline{\mathcal{Z}}_{0},\mathcal{P}^{\natural}\right)\simeq\mathcal{P}_{\left\langle m\right\rangle }^{\simeq}\times\mathcal{P}_{\left\langle 1\right\rangle }^{\simeq}
\]
over $\underline{X}$ is homotopy equivalent to the multi-mapping
space $\mul_{\mathcal{P}}\left(\left\{ X_{1},\dots,X_{m}\right\} ;Y\right)$.
Let $\mathcal{Z}_{0}$ and $\mathcal{Z}_{1}$ be $\infty$-operads
that are fibrant replacements of $\overline{\mathcal{Z}}_{0}$ and
$\overline{\mathcal{Z}}_{1}$, respectively. Moreover, let $f\colon \mathcal{Z}_{0}\to\mathcal{Z}_{1}$
be a map corresponding to the inclusion $\overline{\mathcal{Z}}_{0}\into\overline{\mathcal{Z}}_{1}$.
The functor $F\colon \left(\op\right)_{\mathcal{Z}_{0}/}\to\mathcal{S}$,
co-represented by $f\colon \mathcal{Z}_{0}\to\mathcal{Z}_{1}$, preserves
limits. Furthermore, its value on $g\colon \mathcal{Z}_{0}\to\mathcal{P}$ fits by T.5.5.5.12 into a fiber sequence
\[
\xymatrix{F\left(\mathcal{P}\right)\ar[d]\ar[r] & \map\left(\mathcal{Z}_{1},\mathcal{P}\right)\ar[d]\\
\Delta^{0}\ar[r]^-{\left[g\right]} & \map\left(\mathcal{Z}_{0},\mathcal{P}\right)
,}
\]
which therefore identifies $F\left(\mathcal{P}\right)$ with $\mul_{\mathcal{P}}\left(\left\{ X_{1},\dots,X_{m}\right\} ;Y\right)$
for the objects $X_{1},\dots,X_{m},Y\in\mathcal{P}$ determined by $g$. 

Let $U\colon \oppt\to\left(\op\right)_{\mathcal{Z}_{0}/}$ be the functor induced from the map $\mathcal{Z}_{0}\to\triv$ corresponding to the inclusion $\overline{\mathcal{Z}}_{0}\into\triv$. By T.1.2.13.8, the functor $U$ preserves limits. We define $G^{(m)}\colon \oppt\to\mathcal{S}$ to be the composition of $F$ and $U$, which is limit-preserving as a composition of limit preserving-functors. Unwinding the definitions, $G^{(m)}$ indeed lifts $hG^{(m)}$.
\end{proof}

Let $\mathcal{C}$ be a symmetric monoidal $\infty$-category that
is unital as an $\infty$-operad (ie the unit is an initial object).
For every $X\in\mathcal{C}$ and $m\in\bb N$ we have a canonical
map $\sigma\colon X^{\sqcup m}\to X^{\otimes m}$ defined as follows. For
$k=1,\dots,m$, on the $k$-th summand of $X^{\sqcup m}$ the map
is the tensor product of $m$ maps, where the $k$-th one is $X\oto{\Id}X$
and the rest are the unique map $1_{\mathcal{C}}\to X$. 
\begin{lem}
\label{lem:Reduced_Endomorphism_Mapping_Space} Let $\mathcal{C}$
be a symmetric monoidal $\infty$-category that is unital as an $\infty$-operad and that admits finite coproducts. For every $X\in\mathcal{C}$
and every $m\in\bb N$, there is a fiber sequence
\[
\End_{\mathcal{C}}^{\red}\left(X\right)\left(m\right)\to\map_{\mathcal{C}}\left(X^{\otimes m},X\right)\xrightarrow{\sigma^{*}}\map_{\mathcal{C}}\left(X^{\sqcup m},X\right),
\]
where the fiber is taken over the fold map $\nabla\colon X^{\sqcup m}\to X$.
\end{lem}

\begin{proof}
By \propref{Operad_Reduction_Right_Adjoint} we have a pullback square
of pointed unital $\infty$-operads
\[
\xymatrix{\End_{\mathcal{C}}^{\red}\left(X\right)\ar[r]\ar[d] & \mathcal{C}\ar[d]\\
\bb E_{\infty}\ar[r] & \underline{\mathcal{C}}_{\sqcup}
,}
\]

which, by \lemref{Limits_Operads}, induces a pullback square of multi-mapping
spaces
\[
\xymatrix{\End_{\mathcal{C}}^{\red}\left(X\right)\left(m\right)\ar[r]\ar[d] & \mul_{\mathcal{C}}\left(X^{\left(m\right)},X\right)\ar[d]\\
\bb E_{\infty}\left(m\right)\ar[r] & \mul_{\underline{\mathcal{C}}_{\sqcup}}\left(X^{\left(m\right)},X\right)
.}
\]
The bottom map is the map $\Delta^{0}\to\map_{\mathcal{C}}\left(X^{\sqcup m},X\right)$
that chooses the fold map since it is induced from the map $\bb E_{\infty}=\left(\Delta^{0}\right)_{\sqcup}\to\underline{\mathcal{C}}_{\sqcup}$.
The right vertical map is induced by pre-composition with the 
map $\sigma\colon X^{\sqcup m}\to X^{\otimes m}$, since it is induced by the adjunction
\[
\left(-\right)_{\sqcup}\colon \cat\adj\opun\colon \underline{\left(-\right)}.
\]
\end{proof}

\subsection{Symmetric Sequences}

There is another perspective on reduced $\infty$-operads provided
by the notion of a symmetric sequence. Roughly speaking, a symmetric
sequence is a sequence of $\Sigma_{n}$-spaces $X_{n}$ for $n\ge0$,
where $\Sigma_{n}$ is the symmetric group on $n$ elements. From
an $\infty$-operad $\mathcal{O}$ with an object $X\in\underline{\mathcal{O}}$
one can construct a symmetric sequence of spaces by
\[
\mathcal{O}\left(n\right)=\mul_{\mathcal{O}}( X^{(n)};X),
\]
where the action of $\Sigma_{n}$ comes from permuting the inputs.
For our purposes it is convenient to use the following model:
\begin{defn}
Let $\fin$ denote the skeletal version of the category of finite
sets, ie the	 full subcategory of $\set$ spanned by the objects
$\left[n\right]=\left\{ 1,\dots,n\right\} $ for each integer $n$.
We define the $\infty$-category of symmetric sequences (in spaces),
denoted by $\sseq$, to be $\mathcal{S}_{/\fin^{\simeq}}$. 
\end{defn}

\begin{rem}
We note two things about this definition:
\begin{enumerate}
\item The inclusion of the full subcategory $\mathcal{S}_{/\fin^{\simeq}}^{Kan}\ss\mathcal{S}_{/\fin^{\simeq}}$
spanned by Kan fibrations is an equivalence of $\infty$-categories
and the straightening functor of \cite{HTT} induces an equivalence
of $\infty$-categories $\mathcal{S}_{/\fin^{\simeq}}^{Kan}\simeq\fun\left(\fin^{\simeq},\mathcal{S}\right)$.
Since $\fin^{\simeq}$ is equivalent to the disjoint union of
classifying spaces of the symmetric groups $\Sigma_{n}$, we get 
\[
\sseq\simeq\fun\left(\coprod\limits _{n\ge0}B\Sigma_{n},\mathcal{S}\right)\simeq\prod_{n\ge0}\fun\left(B\Sigma_{n},\mathcal{S}\right).
\]
More explicitly, given a symmetric sequence $p\colon S\to\fin^{\simeq}$, 
taking pullback along the map $\Delta^{0}\to\fin^{\simeq}$ that
corresponds to the object $\left[n\right]\in\fin^{\simeq}$, we obtain
a space $S\left(n\right)$ that is the underlying space of the $\Sigma_{n}$-space on the right hand-side of the above equivalence. 
\item In relating $\infty$-operads to symmetric sequences it is useful
to note that the functor $\fin\to\finpt$, which adds a base point,
induces an isomorphism of groupoids $\fin^{\simeq}\iso\finpt^{\simeq}$.
Moreover, $\finpt^{\simeq}$ is isomorphic to $\triv_{\act}^{\otimes}$
(see the notation in T.3.1.1.1).
\end{enumerate}
\end{rem}

We next define the underlying symmetric sequence of a pointed $\infty$-operad $\mathcal{O}_X$, which is given by a map $\triv\to\mathcal{O}$, such that $X$ is the image of $\left\langle 1\right\rangle \in\triv$ (see \remref{Pointed_Operads}).

\begin{defn}
\label{def:Underlying_SSeq}
Given a pointed $\infty$-operad $\mathcal{O}_X$, we define its underlying symmetric sequence to be 
\[
p\colon \triv_{\act}^{\otimes}\times_{\mathcal{O}_{\act}^{\otimes}}\left(\mathcal{O}_{\act}^{\otimes}\right)_{/X}\to\triv_{\act}^{\otimes}\simeq\fin^{\simeq}
\]
and denote it by $\mathcal{O}_{X,\sseq}$. By analogy with $\infty$-operads, we denote  by $\mathcal{O}_{X,\sseq}^{\otimes}$ the source of $p$.
\end{defn}

\begin{lem}
\label{lem:SSeq_Kan_Fibration} Given a pointed $\infty$-operad $\mathcal{O}_X$, the map
\[
p\colon \triv_{\act}^{\otimes}\times_{\mathcal{O}_{\act}^{\otimes}}\left(\mathcal{O}_{\act}^{\otimes}\right)_{/X}\to\triv_{\act}^{\otimes}\simeq\fin^{\simeq}
\]
is a Kan fibration.
\end{lem}

\begin{proof}
Since $p$ is a pullback of the right fibration $\left(\mathcal{O}_{\act}^{\otimes}\right)_{/X}\to\mathcal{O}_{\act}^{\otimes}$
it is itself a right fibration. The simplicial set $\triv_{\act}^{\otimes}$
is isomorphic to $\fin^{\simeq}$ and is in particular a Kan complex.
By T.2.1.3.3 the map $p$ is a Kan fibration.
\end{proof}

\defref{Underlying_SSeq} relates to the informal description at the beginning of the subsection by

\begin{lem}
\label{lem:P_SSeq_Multi_P} Given a pointed $\infty$-operad $\mathcal{O}_X$, there is a homotopy equivalence
\[
\mathcal{O}_{X,\sseq}(n)\simeq\mul_{\mathcal{O}}(X^{(n)};X),
\]
which is natural in $\mathcal{O}_X$.
\end{lem}

\begin{proof}
This follows directly from unwinding \defref{Underlying_SSeq}.
\end{proof}

Let $f\colon \triv\to\mathcal{O}$ be a pointed $\infty$-operad and let
$p\colon \mathcal{O}\to\mathcal{U}$ be a map of $\infty$-operads. Consider
$\mathcal{U}$ as pointed by the composition $p\circ f$. Let $X=f\left(\left\langle 1\right\rangle \right)$
and $Y=p\left(f\left(\left\langle 1\right\rangle \right)\right)$.
The (1-categorical) functoriality of the formula in \lemref{SSeq_Kan_Fibration} induces a map of symmetric sequences
\[
\triv_{\act}^{\otimes}\times_{\mathcal{O}_{\act}^{\otimes}}\left(\mathcal{O}_{\act}^{\otimes}\right)_{/X}\to\triv_{\act}^{\otimes}\times_{\mathcal{U}_{\act}^{\otimes}}\left(\mathcal{U}_{\act}^{\otimes}\right)_{/Y}.
\]
One can verify that this yields a functor on the level of homotopy categories
\[
	\left(-\right)_{\sseq}\colon h\left(\op\right)_{\triv/}\to h\sseq.
\]

It will be important in what follows to know the following:
\begin{prop}
\label{prop:SSeq_Conservative} The functor $\left(-\right)_{\sseq}$ is conservative.
\end{prop}

\begin{proof}
Let $g\colon \mathcal{P}\to\mathcal{Q}$ be a map of reduced $\infty$-operads
such that $g_{\sseq}$ is an equivalence. The map $g$ is defined
by a commutative triangle
\[
\xymatrix{\mathcal{P}^{\otimes}\ar[dr]\ar[rr]^{g^{\otimes}} &  & \mathcal{Q}^{\otimes}\ar[dl]\\
 & \finpt
.}
\]
To show that $g$ is an equivalence of $\infty$-operads, we need
to show that $g^{\otimes}$ is an equivalence of $\infty$-categories.
Since $\mathcal{P}$ and $\mathcal{Q}$ are reduced, it is clear that
$g^{\otimes}$ is essentially surjective. To show that $g^{\otimes}$ is
fully faithful, we can use the Segal conditions to reduce this to showing
that the map 
\[
\mathcal{P}\left(n\right)=\mul_{\mathcal{P}}(*^{(n)},*)\to\mul_{\mathcal{Q}}(*^{(n)},*)=\mathcal{Q}\left(n\right)
\]
is a homotopy equivalence for all $n$. By \lemref{P_SSeq_Multi_P},
those maps are induced by the equivalence $g_{\sseq}$ and therefore
are equivalences.
\end{proof}

\begin{rem}
It is possible to lift $\left(-\right)_{\sseq}$ to a functor of $\infty$-categories, but a bit tedious to do so. We shall be content with the above weaker version as it will suffice for our applications.
\end{rem}

\subsection{Free Algebras}

The symmetric sequence underlying a reduced $\infty$-operad $\mathcal{P}$
features in the construction of free $\mathcal{P}$-algebras. In what follows we briefly recall and summarize the material of A.3.1.3
specialized to the setting that is of interest to us. That is, let
$\mathcal{P}$ be a reduced $\infty$-operad and let $p\colon \mathcal{C}^{\otimes}\to\finpt$ be a presentably symmetric monoidal $\infty$-category. By A.3.1.3.5
the forgetful functor
\[
U_{\mathcal{P}}\colon \alg_{\mathcal{P}}\left(\mathcal{C}\right)\to\alg_{\triv}\left(\mathcal{C}\right)\simeq\mathcal{C}
\]
admits a left adjoint $F_{\mathcal{P}}$ (the free $\mathcal{P}$-algebra
functor) that can be characterized as follows. By definition A.3.1.3.1,
for every object $X\in\mathcal{C}$ we get a diagram 
	$\mathcal{P}_{\sseq}\left(X\right)\colon \mathcal{P}_{\sseq}^{\otimes}\to\mathcal{\mathcal{C}}_{\act}^{\otimes}$
that, loosely speaking, corresponds to a sequence of maps $\mathcal{P}_{\sseq}\left(n\right)\to\mathcal{\mathcal{C}}_{\act}^{\otimes}$, such that each map lands in the connected component of $X^{\otimes n}$ and is $\Sigma_{n}$-equivariant in the evident way. Furthermore, a map $f\colon X\to U_{\mathcal{P}}\left(A\right)$
in $\mathcal{C}$, gives a lift of $\mathcal{P}_{\sseq}^{\otimes}\left(X\right)$
to a cone diagram 
\[
\mathcal{P}_{\sseq}^{\otimes}\left(f\right)\colon \mathcal{P}_{\sseq}^{\otimes}\to\left(\mathcal{\mathcal{C}}_{\act}^{\otimes}\right)_{/U_{\mathcal{P}}\left(A\right)}.
\]
We say that $f$ exhibits $A$ as the free $\mathcal{P}$-algebra
on $X$, if $\mathcal{P}_{\sseq}^{\otimes}\left(f\right)$ is an operadic
$p$-colimit diagram. By A.3.1.3.2 and A.3.1.3.5, such a map $f$
exists for every $X$ and can be taken as the $X$-component of a
unit natural transformation for an adjunction $F_{\mathcal{P}}\dashv U_{\mathcal{P}}$. 

Using our assumption on $\mathcal{C}$, we can reduce the \emph{operadic}
colimit in the above discussion to an \emph{ordinary} colimit in $\mathcal{C}$.
Consider the following commutative diagram
\[
\xymatrix{\Delta^{\left\{ 0\right\} }\times\mathcal{C}^{\otimes}\ar@{^{(}->}[d]\ar[r]^-{\Id} & \mathcal{C}^{\otimes}\ar[d]^{p}\\
\Delta^{1}\times\mathcal{C}^{\otimes}\ar[r]^-{\alpha}\ar[ru]^{\overline{\alpha}} & \finpt
,}
\]
where $\alpha$ is a natural transformation from $p$ to the constant
diagram on $\left\langle 1\right\rangle $ that consists of active
morphisms. Let $\overline{\alpha}$ be a coCartesian natural transformation
that lifts $\alpha$. The restricted functor $F=\overline{\alpha}|_{\Delta^{\left\{ 1\right\} }\times\mathcal{C}_{\act}^{\otimes}}$
lands in the fiber over $\left\langle 1\right\rangle $ and is therefore
a functor $F\colon \mathcal{C}_{\act}^{\otimes}\to\mathcal{C}$. 
\begin{rem}
Informally speaking, $F$ takes each multi-object $X_{1}\oplus\cdots\oplus X_{n}$
to the tensor product $X_{1}\otimes\cdots\otimes X_{n}$. There are
two abstract characterizations of $F$ (which we shall not use):
\end{rem}

\begin{paraenum}
\item It is the left adjoint of the inclusion $\underline{\mathcal{C}}\into\mathcal{C}_{\act}^{\otimes}$.
\item The symmetric monoidal envelope is a left adjoint to the inclusion
of symmetric monoidal $\infty$-categories into $\infty$-operads.
The functor $F$ is the induced functor on the underlying $\infty$-categories
of the unit of this adjunction at the object $\mathcal{C}$.
\end{paraenum}
By A.3.1.1.15 and A.3.1.1.16, $\mathcal{P}_{\sseq}^{\otimes}\left(f\right)$
is an operadic $p$-colimit diagram if and only if the diagram $\mathcal{P}_{\sseq}\left(f\right)=F\circ\mathcal{P}_{\sseq}^{\otimes}\left(f\right)$
is a colimit diagram in $\mathcal{C}$. In particular, we get
\begin{lem}
\label{lem:Monad_Formula} Let $\mathcal{P}$ be a reduced $\infty$-operad
and let $\mathcal{C}$ be a presentably symmetric monoidal $\infty$-category.
The forgetful functor 
\[
U_{\mathcal{P}}\colon \underline{\alg}_{\mathcal{P}}\left(\mathcal{C}\right)\to\underline{\alg}_{\triv}\left(\mathcal{C}\right)\simeq\mathcal{C}
\]
admits a left adjoint $F_{\mathcal{P}}$ and the associated monad
$T_{\mathcal{P}}=U_{\mathcal{P}}\circ F_{\mathcal{P}}$ acts on an
object $X\in\mathcal{C}$ as follows:
\[
T_{\mathcal{P}}\left(X\right)=U_{\mathcal{P}}F_{\mathcal{P}}\left(X\right)=\colim\mathcal{P}_{\boldsymbol{\sseq}}\left(X\right)=\coprod_{n\ge0}\left(\mathcal{P}\left(n\right)\otimes X^{\otimes n}\right)_{h\Sigma_{n}}
\]
(where we let $\otimes$ denote the canonical enrichment
of $\underline{\mathcal{C}}$ over $\mathcal{S}$ as well).
\end{lem}

Our next goal is to articulate the functoriality of $T_{\mathcal{P}}$
in the $\infty$-operad $\mathcal{P}$. 

\begin{con} \label{con:Map_Monads} Given a map
of reduced $\infty$-operads $\mathcal{P}\to\mathcal{Q}$ we get a
forgetful functor $G\colon \underline{\alg}_{\mathcal{Q}}\left(\mathcal{C}\right)\to\underline{\alg}_{\mathcal{P}}\left(\mathcal{C}\right)$,
such that $U_{\mathcal{P}}G=U_{\mathcal{Q}}$. The unit map 
\[
\Id\to U_{\mathcal{Q}}F_{\mathcal{Q}}=U_{\mathcal{P}}GF_{\mathcal{Q}}
\]
has an adjunct $F_{\mathcal{P}}\to GF_{\mathcal{Q}}$ and by applying
$U_{\mathcal{P}}$ we obtain an induced map of the associated
monads (as endofunctors of $\mathcal{C}$): 
\[
\alpha_{G}\colon T_{\mathcal{P}}=U_{\mathcal{P}}F_{\mathcal{P}}\to U_{\mathcal{P}}GF_{\mathcal{Q}}=U_{\mathcal{Q}}F_{\mathcal{Q}}=T_{\mathcal{Q}},
\]
which is well defined up to homotopy.
\end{con}

\begin{lem}
\label{lem:Equivalence_Gives_Equivalence} In the setting of Construction \conref{Map_Monads},
if $G$ is an equivalence of $\infty$-categories, then the map $\alpha_{G}\colon T_{\mathcal{P}}\to T_{\mathcal{Q}}$ is a natural equivalence of functors.
\end{lem}

\begin{proof}
Since all the steps in the construction are invariant, we may assume
without loss of generality that $G$ is the identity functor and $U_{\mathcal{P}}=U_{\mathcal{Q}}$.
In this case, the map $\alpha_{G}$ is given by applying $U_{\mathcal{Q}}$
to the composition 
\[
F_{\mathcal{Q}}\xrightarrow{F_{\mathcal{Q}} u}F_{\mathcal{Q}}U_{\mathcal{Q}}F_{\mathcal{Q}}\xrightarrow{c F_{\mathcal{Q}}}F_{\mathcal{Q}}
\]
where $u$ and $c$ are the unit and counit of the adjunction $F_{\mathcal{Q}} \dashv U_{\mathcal{Q}}$. This composition is homotopic to the identity by the zig-zag identities.
\end{proof}
Our last task is to show that the map from Construction \conref{Map_Monads}
is induced from the map of symmetric sequences $\mathcal{P}_{\sseq}\to\mathcal{Q}_{\sseq}$
by the functoriality of the explicit formula given in \lemref{Monad_Formula}. 
\begin{lem}
\label{lem:adjunct_cone_map}Given a map $f\colon X\to U_{\mathcal{P}}\left(A\right)$,
the map $\colim\mathcal{P}_{\sseq}\left(X\right)\to U_{\mathcal{P}}\left(A\right)$
induced by the diagram $\mathcal{P}_{\sseq}\left(f\right)$ is equivalent
to the canonical map $\tilde{f}\colon U_{\mathcal{P}}F_{\mathcal{P}}\left(X\right)\to U_{\mathcal{P}}\left(A\right)$
(ie $U_{\mathcal{P}}$ of the adjunct of $f$).
\end{lem}

\begin{proof}
One only has to observe that the map $\tilde{f}\colon U_{\mathcal{P}}F_{\mathcal{P}}\left(X\right)\to U_{\mathcal{P}}\left(A\right)$
is a map of cones on $\mathcal{P}_{\sseq}\left(X\right)$. Let $u_{X}\colon X\to U_{\mathcal{P}}F_{\mathcal{P}}\left(X\right)$
be the unit map of the free-forgetful adjunction at $X$. The adjunct
map $F_{\mathcal{P}}\left(X\right)\to A$ induces a map $\tilde{f}^{\triangleright}\colon \left(\mathcal{\mathcal{C}}_{\act}^{\otimes}\right)_{/U_{\mathcal{P}}F_{\mathcal{P}}\left(X\right)}\to\left(\mathcal{\mathcal{C}}_{\act}^{\otimes}\right)_{/U_{\mathcal{P}}\left(A\right)}$.
Inspecting Construction A.3.1.3.1, it can be seen that the cone diagram $\mathcal{P}_{\sseq}\left(f\right)$
is equivalent to the composition of the universal cone diagram $\mathcal{P}_{\sseq}\left(u_{X}\right)$
and $\tilde{f}^{\triangleright}$.
\end{proof}
From this we get
\begin{prop}
\label{prop:Monad_Map_Formula}Let $g\colon \mathcal{P}\to\mathcal{Q}$
be a map of reduced $\infty$-operads and let $\mathcal{C}$ be a presentably symmetric monoidal $\infty$-category. For every object $X\in\mathcal{C}$,
the induced map of the associated monads 
\[
T_{\mathcal{P}}\left(X\right)=\colim\mathcal{P}_{\sseq}\left(X\right)\to\colim\mathcal{Q}_{\sseq}\left(X\right)=T_{\mathcal{Q}}\left(X\right)
\]
is equivalent to the canonical map on colimits that is induced by pre-composition with
\[
g_{\sseq}\colon \mathcal{P}_{\sseq}\to\mathcal{Q}_{\sseq}.
\]
\end{prop}

\begin{proof}
We denote by $G\colon \underline{\alg}_{\mathcal{Q}}\left(\mathcal{C}\right)\to\underline{\alg}_{\mathcal{P}}\left(\mathcal{C}\right)$
the forgetful functor induced by the map $g$. Let 
\[
f\colon X\to U_{\mathcal{Q}}F_{\mathcal{Q}}\left(X\right)\simeq U_{\mathcal{P}}GF_{\mathcal{Q}}\left(X\right)
\]
be the unit map. It induces a cone diagram 
\[
\mathcal{P}_{\sseq}\left(f\right)\colon \mathcal{P}_{\sseq}\to\mathcal{\mathcal{C}}_{/U_{\mathcal{Q}}F_{\mathcal{Q}}\left(X\right)}
\]
and, by \lemref{adjunct_cone_map}, the associated map $\tilde{f}\colon U_{\mathcal{P}}F_{\mathcal{P}}\left(X\right)\to U_{\mathcal{Q}}F_{\mathcal{Q}}\left(X\right)$
is equivalent to the map $\colim\mathcal{P}_{\sseq}\left(X\right)\to U_{\mathcal{Q}}F_{\mathcal{Q}}\left(X\right)$
specified by the cone diagram $\mathcal{P}_{\sseq}\left(f\right)$.
On the other hand, inspecting Construction A.3.1.3.1, it can be seen that the diagram
$\mathcal{P}_{\sseq}\left(f\right)$ is obtained from the diagram
\[
\mathcal{Q}_{\sseq}\left(f\right)\colon \mathcal{Q}_{\sseq}\to\mathcal{\mathcal{C}}_{/U_{\mathcal{Q}}F_{\mathcal{Q}}\left(X\right)}
\]
by pre-composition with $g_{\sseq}\colon \mathcal{P}_{\sseq}\to\mathcal{Q}_{\sseq}$
and that $\mathcal{Q}_{\sseq}\left(f\right)$ exhibits $U_{\mathcal{Q}}F_{\mathcal{Q}}\left(X\right)$
as the colimit of $\mathcal{Q}_{\sseq}\left(X\right)$. Thus, we get
the desired equivalence.
\end{proof}


\section{$d$-Categories and $d$-Operads}

This section deals with \emph{essentially $d$-categories}, ie $\infty$-categories all of whose mapping spaces are $\left(d-1\right)$-truncated (\defref{Ess_d_Category}), and with the analogous notion for $\infty$-operads (\defref{Ess_d_Operad}). 

In 3.1 we discuss the fact that the inclusion of the full subcategory spanned by the essentially $d$-categories (resp. $d$-operads) into $\cat$ (resp. $\op$) admits a left adjoint and that the unit of this adjunction consists of $(d-1)$-truncation of the (multi)-mapping spaces. The proofs of these (very plausible) facts are rather technical, involving a combinatorial analysis of some strict models for the above constructions, and can be found in \cite{SY19}.
In 3.2 we use the results of 3.1 to characterize when a map of $\infty$-operads induces an equivalence on $d$-homotopy operads in terms of the induced functor on algebras in a $d$-topos (\propref{d-equivalence}).

\subsection{$d$-Homotopy Categories and Operads}

Recall the following definition from classical homotopy theory.
\begin{defn}
\label{def:Truncated_Space}For $d\ge0$, a space $X\in\mathcal{S}$
is called $d$-\emph{truncated} if $\pi_{i}\left(X,x\right)=0$ for all $i>d$
and all $x\in X$. 
In addition, a space is called $\left(-2\right)$-\emph{truncated} if and only if it is contractible and it is called $\left(-1\right)$-\emph{truncated} if and only if it is either contractible or empty. 
We denote by $\mathcal{S}_{\le d}$ the full subcategory of $\mathcal{S}$ spanned by the $d$-truncated spaces. 
The inclusion $\mathcal{S}_{\le d}\into\mathcal{S}$ admits a left adjoint and we call the unit of the adjunction the $d$-\emph{truncation map}.
\end{defn}

This leads to the following definition in $\infty$-category theory.
\begin{defn}
\label{def:Ess_d_Category}Let $d\ge-1$ be an integer. An \emph{essentially
$d$-category} is an $\infty$-category $\mathcal{C}$ such that for
all $X,Y\in\mathcal{C}$, the mapping space $\map_{\mathcal{C}}\left(X,Y\right)$
is $\left(d-1\right)$-truncated. We denote by $\catd$ the full subcategory
of $\cat$ spanned by essentially $d$ -categories. 
\end{defn}

\begin{rem}
An $\infty$-category $\mathcal{C}$ is an essentially $1$-category
if and only if it lies in the essential image of the nerve functor
$N\colon \mathbf{Cat}\to\cat$ and it is an essentially $0$-category if
and only if it is equivalent to the nerve of a poset.
\end{rem}

In T.2.3.4, Lurie develops the theory of $d$-categories (see definition
T.2.3.4.1), which are a strict model for essentially $d$-categories. In particular, he associates with every $\infty$-category $\mathcal{C}$, a $d$-category $h_d\mathcal{C}$ (see Proposition T.2.3.4.12), which we refer to as the \emph{$d$-homotopy category} of $\mathcal{C}$.
In \cite{SY19} we make a further study of this theory and use it to prove the following:

\begin{prop}[{\cite[Theorem 2.15]{SY19}}]
\label{prop:d-homotopy_category}The inclusion $\catd\into\cat$ admits
a left adjoint $h_{d}$, such that for every $\infty$-category $\mathcal{C}$, the value of $h_d$ on $\mathcal{C}$  is the $d$-homotopy category of $\mathcal{C}$, the unit transformation $\theta_{d}\colon \mathcal{C}\to h_{d}\mathcal{C}$
is essentially surjective, and for all $X,Y\in\mathcal{C}$, the map of spaces
\[
\map_{\mathcal{C}}\left(X,Y\right)\to\map_{h_{d}\mathcal{C}}\left(\theta_{d}\left(X\right),\theta_{d}\left(Y\right)\right)
\]
is the $\left(d-1\right)$-truncation map.
\end{prop}

\begin{war}
Note that an $\infty$-category \emph{$\mathcal{C}$} is an essentially
$d$-category if and only if all objects of $\mathcal{C}$ are $\left(d-1\right)$-truncated in the sense of T.5.5.6.1. 
Hence, another way to associate an essentially $d$-category with an $\infty$-category $\mathcal{C}$ is to consider the full subcategory spanned by the $\left(d-1\right)$-truncated objects. 
For a presentable $\infty$-category, this is denoted by $\tau_{\le d-1}\mathcal{C}$ in T.5.5.6.1 and called the \emph{$\left(d-1\right)$-truncation of} $\mathcal{C}$. 
We warn the reader that the two essentially $d$-categories $h_{d}\mathcal{C}$ and $\tau_{\le d-1}\mathcal{C}$ are usually very different. 
For example, when $\mathcal{C}=\mathcal{S}$ is the $\infty$-category of spaces, $h_{1}\mathcal{S}$ is the ordinary homotopy category of spaces, while $\tau_{\le0}\mathcal{S}$ is equivalent to the ordinary category of sets. 
Both constructions will play a central role in the proof of the main result, and hopefully the distinction in notation and terminology will prevent confusion.

With these ideas in mind, one might hope that for an $\infty$-category $\mathcal{C}$, the condition of being an essentially $(d+1)$-category would coincide with the condition of begin a $d$-truncated object of the presentable $\infty$-category $\cat$. 
This turns out to be \emph{false}. More precisely, it can be shown that a $d$-truncated object of $\cat$ is an essentially $(d+1)$-category and that an essentially $(d+1)$-category is a $(d+1)$-truncated object of $\cat$, but neither of the converses hold (see \cite[Remark 2.10]{SY19}).
\end{war}

By analogy with the above, we also have a natural notion of an essentially $d$-operad.

\begin{defn}\label{def:Ess_d_Operad}
Let $d\ge-1$. An \emph{essentially $d$-operad} is an $\infty$-operad
$\mathcal{O}$ such that for all $X_{1},\dots,X_{n},Y\in\underline{\mathcal{O}}$,
the multi-mapping space $\mbox{Mul}_{\mathcal{O}}\left(\left\{ X_{1},\dots,X_{n}\right\} ;Y\right)$
is $\left(d-1\right)$-truncated. We denote by $\opd$ the full subcategory
of $\op$ spanned by essentially $d$ -operads.
\end{defn}

\begin{example}
\label{exa:Sym_Mon_d_Cat} Two important special cases are:
\begin{paraenum}
\item A symmetric monoidal $\infty$-category $\mathcal{C}$ is an essentially
$d$-operad if and only if the underlying $\infty$-category $\underline{\mathcal{C}}$
is an essentially $d$-category.
\item A reduced $\infty$-operad $\mathcal{P}$ is an essentially $d$-operad
if and only if the symmetric sequence $\left\{ \mathcal{P}\left(n\right)\right\} _{n\ge0}$
consists of $\left(d-1\right)$-truncated spaces.
\end{paraenum}
\end{example}

In \cite{SY19} we develop a parallel notion of a \emph{$d$-operad}, that bears the same relation to an essentially $d$-operad as a $d$-category does to an essentially $d$-category; ie it is a strict model for an essentially $d$-operad. Using this theory we show the following:

\begin{prop}[{\cite[Theorem 3.12]{SY19}}]
\label{prop:d-homotopy_operad} The inclusion $\opd\into\op$ admits
a left adjoint $h_{d}$, such that for every $\infty$-operad $\mathcal{O}$,
the unit transformation $\theta_{d}\colon \mathcal{O}\to h_{d}\mathcal{O}$ is essentially surjective and for all
$X_{1},\dots,X_{n},Y\in \underline{\mathcal{O}}$,
the map of spaces
\[
\mul_{\mathcal{O}}\left(\left\{ X_{1},\dots,X_{n}\right\} ;Y\right)\to\mul_{h_{d}\mathcal{O}}\left(\left\{ \theta_{d}\left(X_{1}\right),\dots,\theta_{d}\left(X_{n}\right)\right\} ;\theta_{d}\left(Y\right)\right)
\]
is the $\left(d-1\right)$-truncation map.
\end{prop}

\begin{defn}
Given an $\infty$-operad $\mathcal{O}$, we refer to $h_d\mathcal{O}$, as the \emph{$d$-homotopy operad} of $\mathcal{O}$.
\end{defn}

For future use, we record the following fact:
\begin{prop}[{\cite[Proposition 3.13]{SY19}}]
\label{prop:Alg_d_Category} Let $\mathcal{O}$ be an $\infty$-operad
and let $\mathcal{U}$ be an essentially $d$-operad. The $\infty$-category
$\underline{\alg}_{\mathcal{O}}\left(\mathcal{U}\right)$ is an essentially
$d$-category.
\end{prop}

\subsection{$d$-Equivalences and $d$-Topoi}

\begin{defn}
\label{def:d_Equivalence}For $d\ge-2$, a map of $\infty$-operads
$f\colon \mathcal{O}\to\mathcal{U}$ is called a \emph{$d$-equivalence}, if the induced
map $h_{d+1}\left(f\right)\colon h_{d+1}\mathcal{O}\to h_{d+1}\mathcal{U}$
is an equivalence of $\infty$-operads, ie if it is essentially surjective
on the underlying categories and induces an equivalence on the $d$-truncations of all the multi-mapping spaces.
\end{defn}

An important special case is
\begin{defn}
\label{def:d_Connectedness}For $d\ge-2$, an $\infty$-operad $\mathcal{O}$
is called $d$-connected if the unique map from $\mathcal{O}$ to
the terminal $\infty$-operad $\bb E_{\infty}$ is a $d$-equivalence, 
ie if all the multi-mapping spaces in $\mathcal{O}$ are $d$-connected.
\end{defn}

\begin{rem}
\label{rem:Minus_One_Connected}Let $\mathcal{P}$ be a reduced $\infty$-operad.
It is $d$-connected if and only if all the spaces 
$\mathcal{P}\left(n\right)$
in the underlying symmetric sequence of $\mathcal{P}$ are $d$-connected.
If $\mathcal{P}$ is not equivalent to $\bb E_{0}$, then for some
$n\ge2$ we have $\mathcal{P}\left(n\right)\neq\es$, and so there exists an
$n$-ary operation $\mu\in\mathcal{P}\left(n\right)$ for $n\ge2$.
By composing $\mu$ with itself, we can obtain an operation in $\mathcal{P}$
of arbitrarily high arity and by composition with the unique nullary
operation, we can obtain an operation of arbitrary arity. It follows
that $\mathcal{P}\not\simeq\bb E_{0}$ if and only if $\mathcal{P}$
is $\left(-1\right)$-connected.
\end{rem}

The main result of this section is a characterization of $d$-equivalences of reduced $\infty$-operads. But first, we need some preliminary observations about Cartesian symmetric monoidal structures. 
\begin{lem}
\label{lem:Cartesian_conservative}Let $f_{\alpha}\colon \mathcal{D}\to\mathcal{C}_{\alpha}$
be a collection of jointly conservative, symmetric monoidal functors
between symmetric monoidal $\infty$-categories.
\begin{lemenum}
\item If $\mathcal{C}_{\alpha}$ is Cartesian and $f_{\alpha}$ preserves
finite products for all $\alpha$ and $\underline{\mathcal{D}}$ has
all finite products, then $\mathcal{D}$ is Cartesian.
\item If $\mathcal{C}_{\alpha}$ is coCartesian and $f_{\alpha}$ preserves
finite coproducts for all $\alpha$ and $\underline{\mathcal{D}}$
has all finite coproducts, then $\mathcal{D}$ is coCartesian.
\end{lemenum}
\end{lem}

\begin{proof}
By A.2.4.2.7, the opposite of a symmetric monoidal $\infty$-category
acquires a symmetric monoidal structure, which is Cartesian if and
only if the original symmetric monoidal $\infty$-category is coCartesian.
Hence, it is enough to prove (2). The unit object $1\in \mathcal{D}$
has a unique map from the initial object $\es\to1$. Since $f_{\alpha}$
is both symmetric monoidal and preserves finite coproducts, $f_{\alpha}\left(\es\to1\right)$ is the unique map from the initial object to the unit object of $\mathcal{C}_{\alpha}$, which is an equivalence by assumption. Since the collection of $f_{\alpha}$
is jointly conservative, it follows that the unit of $\mathcal{D}$
is initial in $\underline{\mathcal{D}}$ as well. Namely, $\mathcal{D}$
is unital as an $\infty$-operad. Using \lemref{coCart_Left_Adjoint}
we have a map of $\infty$-operads $G\colon \mathcal{D}\to\mathcal{D}_{\sqcup}$, which is an equivalence on the underlying $\infty$-categories. We
need to show that this map is symmetric monoidal. Namely, that it
maps coCartesian edges (over $\finpt$) to coCartesian edges. Since
we already know that it is a map of $\infty$-operads and hence preserves
inert morphisms, we only need to show that active coCartesian edges
map to coCartesian edges. Using the Segal conditions, we are further
reduced to considering only coCartesian lifts of the unique active
morphism $\mu\colon \left\langle n\right\rangle \to\left\langle 1\right\rangle $.
For every collection of objects $X_{1},X_{2},\dots,X_{n}\in\underline{\mathcal{D}}$, let 
\[
\mu_{\otimes}\colon X_{1}\oplus\cdots\oplus X_{n}\to X_{1}\otimes X_{2}\otimes\dots\otimes X_{n}
\]
be a coCartesian lift of $\mu$ to $\mathcal{D}^{\otimes}$. Since
$G$ is an equivalence on the underlying $\infty$-categories, $G\left(\mu_{\otimes}\right)$
can be considered as a map
\[
X_{1}\oplus\cdots\oplus X_{n}\to X_{1}\otimes X_{2}\otimes\dots\otimes X_{n}
\]
in $\mathcal{D}_{\sqcup}$. Now, let 
\[
\mu_{\sqcup}\colon X_{1}\oplus\cdots\oplus X_{n}\to X_{1}\sqcup X_{2}\sqcup\dots\sqcup X_{n}
\]
be a coCartesian lift of $\mu$ to $\mathcal{D}^{\sqcup}$. There
exists a unique (up to homotopy) map
\[
g_{X_{1},\dots,X_{n}}\colon X_{1}\sqcup X_{2}\sqcup\dots\sqcup X_{n}\to X_{1}\otimes X_{2}\otimes\dots\otimes X_{n},
\]
such that $G\left(\mu_{\otimes}\right)=g_{X_{1},\dots,X_{n}}\circ\mu_{\sqcup}$.
We need to show that $g_{X_{1},\dots,X_{n}}$ is an equivalence in
$\underline{\mathcal{D}}$ for all $X_{1},\dots,X_{n}\in\underline{\mathcal{D}}$.
For every $\alpha$, we have a homotopy commutative diagram 
\[
\xymatrix{\mathcal{D}^{\otimes}\ar[r]\ar[d]^{f_{\alpha}} & \mathcal{D}^{\sqcup}\ar[d]^{f_{\alpha}}\\
\mathcal{C}_{\alpha}^{\otimes}\ar[r] & \mathcal{C}_{\alpha}^{\sqcup}
}
\]
in which the vertical and bottom maps are symmetric monoidal.
It follows that $f_{\alpha}\left(g_{X_{1},\dots,X_{n}}\right)$ is
an equivalence in $\mathcal{C}_{\alpha}$ for all $\alpha$. By joint
conservativity, $g_{X_{1},\dots,X_{n}}$ is an equivalence as well.
\end{proof}
\begin{lem}
\label{lem:Cart_Sym_Mon_Strcture}Let $\mathcal{C}_{\times}$ be a
Cartesian symmetric monoidal $\infty$-category. For every $\infty$-operad
$\mathcal{D}$, the $\infty$-operad $\alg_{\mathcal{D}}\left(\mathcal{C}\right)$
(see A.2.2.5.4) is also Cartesian.
\end{lem}

\begin{proof}
By A.2.2.5.4, since $\mathcal{C}_{\times}$ is symmetric monoidal,
so is $\alg_{\mathcal{D}}\left(\mathcal{C}\right)$ and, for every $X\in\mathcal{D}$, the evaluation functor $e_{X}\colon \alg_{\mathcal{D}}\left(\mathcal{C}\right)\to\mathcal{C}_{\times}$
is a symmetric monoidal functor. On the underlying $\infty$-categories, $e_{X}$
also preserves finite products since it preserves all limits. Finally,
we show that the collection of evaluation functors is jointly conservative
since they can be presented as the composition of the conservative
restriction functor
\[
\alg_{\mathcal{D}}\left(\mathcal{C}\right)\to\fun\left(\mathcal{D},\mathcal{C}\right)
\]
and the collection of evaluation functors
\[
e_{X}\colon \fun\left(\mathcal{D},\mathcal{C}\right)\to\mathcal{C},
\]
which are jointly conservative by T.3.1.2.1. Now, by \lemref{Cartesian_conservative}(1),
$\alg_{\mathcal{D}}\left(\mathcal{C}\right)$ is Cartesian.
\end{proof}

We are now ready for the main proposition.
\begin{prop}
\label{prop:d-equivalence} Let $d \ge -1$. Given a map of reduced $\infty$-operads $f\colon \mathcal{P}\to\mathcal{Q}$, the following are equivalent:
\begin{lemenum}
\item The map $f$ is a $d$-equivalence.
\item For every $\left(d+1\right)$-topos $\mathcal{C}$, the induced map
\[
\map_{\op}\left(\mathcal{Q},\mathcal{C}_{\times}\right)\to\map_{\op}\left(\mathcal{P},\mathcal{C}_{\times}\right)
\]
is a homotopy equivalence.
\item For every simplicial set $K$, the induced map 
\[
\map_{\op}\left(\mathcal{Q},\mathcal{S}_{\le d}^{K}\right)\to\map_{\op}\left(\mathcal{P},\mathcal{S}_{\le d}^{K}\right)
\]
is a homotopy equivalence where $\mathcal{S}_{\le d}^{K}$ is given
the Cartesian symmetric monoidal structure.
\item The induced map 
\[
\underline{\alg}_{\mathcal{Q}}\left(\mathcal{S}_{\le d}\right)\to\underline{\alg}_{\mathcal{P}}\left(\mathcal{S}_{\le d}\right)
\]
is an equivalence of $\infty$-categories, where $\mathcal{S}_{\le d}$
is given the Cartesian symmetric monoidal structure.
\end{lemenum}
\end{prop}

\begin{proof}

$\left(1\right)\implies\left(2\right)$ Consider the commutative diagram
\[
\xymatrix{\map_{\op}\left(\mathcal{Q},\mathcal{C}\right)\ar[d]\ar[r] & \map_{\op}\left(\mathcal{P},\mathcal{C}\right)\ar[d]\\
\map_{\op}\left(h_{d+1}\mathcal{Q},\mathcal{C}\right)\ar[r] & \map_{\op}\left(h_{d+1}\mathcal{P},\mathcal{C}\right)
.}
\]

Since $h_{d+1}\left(\mathcal{P}\right)\to h_{d+1}\left(\mathcal{Q}\right)$
is an equivalence of $\infty$-operads, the bottom map is a homotopy
equivalence. By \propref{d-homotopy_operad}, the vertical
maps are equivalences as well, and so, by the 2-out-of-3 property, the top map is an equivalence.

$\left(2\right)\implies\left(3\right)$ Since $\mathcal{S}_{\le d}^{K}$
is a $\left(d+1\right)$-topos, this is just a special case.

$\left(3\right)\implies\left(4\right)$ By Yoneda's lemma applied
to $\cat$, the map 
\[
\underline{\alg}_{\mathcal{Q}}\left(\mathcal{S}_{\le d}\right)\to\underline{\alg}_{\mathcal{P}}\left(\mathcal{S}_{\le d}\right)
\]
is an equivalence of $\infty$-categories if for every $\infty$-category
$\mathcal{E}$, the map 
\[
\map_{\cat}(\mathcal{E},\underline{\alg}_{\mathcal{Q}}\left(\mathcal{S}_{\le d}\right))\to\map_{\cat}(\mathcal{E},\underline{\alg}_{\mathcal{P}}\left(\mathcal{S}_{\le d}\right))
\]
is a homotopy equivalence. Using the fully faithful embedding $\cat\into\op$, which is left adjoint to the underlying category functor $\op\to\cat$
(see A.2.1.4.11), this map is equivalent to 
\[
\map_{\op}\left(\mathcal{E},\alg_{\mathcal{Q}}\left(\mathcal{S}_{\le d}\right)\right)\to\map_{\op}\left(\mathcal{E},\alg_{\mathcal{P}}\left(\mathcal{S}_{\le d}\right)\right).
\]
By adjointness with the Boardman\textendash Vogt tensor product and the fact
that it is symmetric, the map is equivalent to 
\[
\map_{\op}\left(\mathcal{Q},\alg_{\mathcal{E}}\left(\mathcal{S}_{\le d}\right)\right)\to\map_{\op}\left(\mathcal{P},\alg_{\mathcal{E}}\left(\mathcal{S}_{\le d}\right)\right).
\]
Since $\mathcal{E}$ is an $\infty$-category, by \lemref{Cart_Sym_Mon_Strcture} the $\infty$-operad $\alg_{\mathcal{E}}\left(\mathcal{S}_{\le d}\right)$ is just the $\infty$-category of functors $\left(\mathcal{S}_{\le d}\right)^{\mathcal{E}}$ endowed with the Cartesian symmetric monoidal structure. 
Since the functor category is invariant under Joyal equivalences, we can replace $\mathcal{E}$ with any simplicial set $K$.

$\left(4\right)\implies\left(1\right)$ Consider the commutative diagram
\[
\xymatrix{\underline{\alg}_{\mathcal{Q}}\left(\mathcal{S}_{\le d}\right)\ar[d]^{\wr}\ar[r] & \underline{\alg}_{\mathcal{P}}\left(\mathcal{S}_{\le d}\right)\ar[d]^{\wr}\\
\underline{\alg}_{h_{d}\mathcal{Q}}\left(\mathcal{S}_{\le d}\right)\ar[r] & \underline{\alg}_{h_{d}\mathcal{P}}\left(\mathcal{S}_{\le d}\right)
.}
\]
By \propref{d-homotopy_operad}, the vertical maps are equivalences;
hence by 2-out-of-3, the top map is an equivalence if and only if
the bottom map is. We can therefore assume without loss of generality
that $\mathcal{P}$ and $\mathcal{Q}$ are themselves essentially
$d$-operads. This implies that $\mathcal{P}\left(n\right)$ and $\mathcal{Q}\left(n\right)$
are $d$-truncated spaces for all $n\ge0$. Now, consider the commutative
diagram 
\[
\xymatrix{\underline{\alg}_{\mathcal{Q}}\left(\mathcal{S}_{\le d}\right)\ar[rr]^{f^{*}}\ar[rd]_{U_{\mathcal{Q}}} &  & \underline{\alg}_{\mathcal{P}}\left(\mathcal{S}_{\le d}\right)\ar[ld]^{U_{\mathcal{P}}}\\
 & \mathcal{S}_{\le d}
,}
\]
where $U_{\mathcal{P}}$ and $U_{\mathcal{Q}}$ are the corresponding
forgetful functors. By \lemref{Equivalence_Gives_Equivalence}, the
associated map 
\[
T_{\mathcal{P}}=\coprod_{n}\left(\mathcal{P}\left(n\right)\times X^{n}\right)_{h\Sigma_{n}}\iso\coprod_{n}\left(\mathcal{Q}\left(n\right)\times X^{n}\right)_{h\Sigma_{n}}=T_{\mathcal{Q}}
\]
of Construction \conref{Map_Monads} is a natural equivalence of functors.
On the other hand, by \propref{Monad_Map_Formula}, this map is induced
from a map of symmetric sequences $f_{\sseq}\colon \left\{ \mathcal{P}\left(n\right)\right\} \to\left\{ \mathcal{Q}\left(n\right)\right\} $.
We want to deduce that $f_{\sseq}$ is an equivalence. For $d=-1$,
there is nothing to prove and so we assume that $d\ge0$. Taking $X=\left[n\right]$, there is a coproduct decomposition 
\[
\left(\mathcal{P}\left(n\right)\times X^{n}\right)_{h\Sigma_{n}}=\mathcal{P}\left(n\right)\sqcup J,
\]
where the summand $\mathcal{P}\left(n\right)$ corresponds to orbits
of points whose $X^{n}$ component is a permutation (note that when
$d=0$, the homotopy orbits in $\mathcal{S}_{\le0}$ are just the
orbits as a set). This characterization implies that $f_{\sseq}\colon \mathcal{P}\left(n\right)\to\mathcal{Q}\left(n\right)$
is an equivalence. Finally, since $\left(-\right)_{\sseq}$ is conservative,
by \propref{SSeq_Conservative}, we deduce that $f$ is an equivalence.
\end{proof}

\section{Truncatedness and Connectedness}

This section deals with properties of truncated and connected morphisms
in a presentable $\infty$-category. We begin in 4.1 with some basic
facts about the space of lifts in a commutative square. The key result
is \propref{Lifts_Diagonal}, which expresses the homotopy fiber of
the diagonal of the space of lifts as the space of lifts in a closely
related square. In 4.2 we expand on the notions of $n$-truncated
and $n$-connected morphisms. The main result is \propref{Lifts_Space_Truncatedness},
which is a quantitative version of the defining orthogonality relation
between $n$-connected and $n$-truncated morphisms. In 4.3 we introduce
an auxiliary notion of an $\left(n-\frac{1}{2}\right)$-connected
morphism and compare it with the notion of an $n$-connected morphism
under some assumptions on the ambient $\infty$-category. We conclude
with 4.4 in which we study the notion of $n$-connectedness for the
$\infty$-category of algebras over a reduced $\infty$-operad. In
particular, we show that under some reasonably general conditions,
a map of algebras is $n$-connected if the map between the underlying
objects is $n$-connected (\propref{Algebraic_Reduction}). 

We rely on T.5.5.6 for the basic theory of truncated morphisms and
objects, but we note that the properties of connected morphisms are
studied in \cite{HTT} only in the context of $\infty$-topoi. Some
further results, still in the context of $\infty$-topoi, can be found
in \cite{ABFJ17}. For example, our \propref{Lifts_Space_Truncatedness}
is a generalization of Proposition 3.15 of \cite{ABFJ17}
from $\infty$-topoi to general presentable $\infty$-categories (such
as the $\infty$-category of algebras over an $\infty$-operad). Some
results on truncatedness and connectedness for general presentable
$\infty$-categories can also be found in \cite{GK17}. In
fact, \lemref{Connectedness_Adjoints} and \lemref{connected_via_truncation}
(with its corollary) already appear in \cite{GK17}, yet we
have chosen to include detailed proofs for completeness. Though we shall
not use it, it is worthwhile to mention another result from \cite{GK17}, namely, that the pair of classes of $n$-connected and $n$-truncated morphisms form a factorization system for every presentable $\infty$-category $\mathcal{C}$ (generalizing T.5.2.8.16. from $\infty$-topoi). 

We reiterate that, especially in this section, some of the facts that
we state as lemmas might appear obvious or well known. Nonetheless,
we have chosen to include detailed proofs where those are not to be found
in the literature (to the best of our knowledge). 

\subsection{Space of Lifts}
\begin{defn}
\label{def:Space_of_Lifts}(T.5.2.8.1) A commutative square in an
$\infty$-category $\mathcal{C}$ is a map $q\colon \Delta^{1}\times\Delta^{1}\to\mathcal{C}$,
which we write somewhat informally as 
\[
\Square {A}{B}{X}{Y,}
\]
suppressing the homotopies. The space of lifts for $q$ is defined
as follows. Restricting to the diagonal $\Delta^{1}\to\Delta^{1}\times\Delta^{1}$,
we get a morphism $h\colon A\to Y$ in $\mathcal{C}$, which can be viewed
as an object $\overline{Y}$ in the $\infty$-category $\mathcal{C}_{A/}$.
The diagram $q$ can be encoded as a pair of objects $B,X\in\mathcal{C}_{A//\overline{Y}}$
and the space of lifts for $q$ is given as the mapping space 
\[
L\left(q\right)=\map_{\mathcal{C}_{A//\overline{Y}}}\left(\overline{B},\overline{X}\right).
\]
\end{defn}

\begin{rem}
\label{rem:Lifts_Equiv_Defs} Let us denote the horizontal morphisms
in the above diagram by $f\colon A\to X$ and $g\colon B\to Y$. By the dual of T.5.5.5.12 we have a homotopy fiber sequence
\[
\map_{\mathcal{C}_{A//\overline{Y}}}\left(B,X\right)\to\map_{\mathcal{C}_{A/}}\left(B,X\right)\to\map_{\mathcal{C}_{A/}}\left(B,Y\right)
\]
over $g\in\map_{\mathcal{C}_{A/}}\left(B,Y\right)$. Using T.5.5.5.12
again for the middle and the right term we obtain a presentation of
$\map_{\mathcal{C}_{A//\overline{Y}}}\left(B,X\right)$ as the total
fiber of the square
\[
\xymatrix{\map_{\mathcal{C}}\left(B,X\right)\ar[d]\ar[r] & \map_{\mathcal{C}}\left(B,Y\right)\ar[d]\\
\map_{\mathcal{C}}\left(A,X\right)\ar[r] & \map_{\mathcal{C}}\left(A,Y\right)
.}
\]
In other words, we have a homotopy fiber sequence 
\[
L\left(q\right)\to\map_{\mathcal{C}}\left(B,X\right)\to\map_{\mathcal{C}}\left(A,X\right)\times_{\map_{\mathcal{C}}\left(A,Y\right)}^{h}\map_{\mathcal{C}}\left(B,Y\right)
\]
over the point determined by the diagram $q$.

Another reasonable definition of the space of lifts is as follows.
The inclusion $\Delta^{\left\{ 0,1\right\} }\times\Delta^{\left\{ 0,2\right\} }\into\Delta^{3}$
induces a restriction map $\mathcal{C}^{\Delta^{3}}\to\mathcal{C}^{\Delta^{1}\times\Delta^{1}}$
and we can consider the (automatically homotopy) fiber over the vertex
$q\in\mathcal{C}^{\Delta^{1}\times\Delta^{1}}$, which is an $\infty$-category.
In T.5.2.8.22 it is proved that this $\infty$-category is categorically
equivalent to $L\left(q\right)$ (and in particular a Kan complex). 
\end{rem}

The next lemma shows that the space of lifts behaves well with respect
to pullback and pushout.
\begin{lem}
\label{lem:Lifts_Pullback_Pushout} Given a commutative rectangle
$\Delta^{1}\times\Delta^{2}\to\mathcal{C}$, depicted as
\[
\Rect ABXYZ{W,}
\]
with left square $q_{l}$, right square $q_{r}$, and outer square
$q$,
\begin{lemenum}
\item If $q_{r}$ is a pullback square, then we have a canonical equivalence
$L\left(q\right)\simeq L\left(q_{l}\right)$. 
\item If $q_{l}$ is a pushout square, then we have a canonical equivalence
$L\left(q\right)\simeq L\left(q_{r}\right)$. 
\end{lemenum}
\end{lem}

\begin{proof}
By symmetry, it is enough to prove (1). Observe that the prism $\Delta^{1}\times\Delta^{2}$
is a left cone on the simplicial set obtained by removing the initial
vertex. Formally,

\[
\Delta^{1}\times\Delta^{2}\simeq\left(\Delta^{2}\times\Delta^{\left\{ 1\right\} }\sqcup_{\Delta^{\left\{ 1,2\right\} }\times\Delta^{\left\{ 1\right\} }}\Delta^{\left\{ 1,2\right\} }\times\Delta^{1}\right)^{\triangleleft}.
\]
We can therefore interpret the rectangle as a diagram in $\mathcal{C}_{A/}$
(and hence ignore $A$). Since the projection $\mathcal{C}_{A/}\to\mathcal{C}$
preserves and reflects limits (dual of T.1.2.13.8), the square $q_{r}$
is a pullback square in $\mathcal{C}_{A/}$. The universal property
of the pullback implies that we have a homotopy Cartesian square
\[
\Square{\map_{\mathcal{C}_{A/}}\left(B,X\right)}{\map_{\mathcal{C}_{A/}}\left(B,Y\right)}{\map_{\mathcal{C}_{A/}}\left(B,Z\right)}{\map_{\mathcal{C}_{A/}}\left(B,W\right),}
\]
which in turn induces a homotopy equivalence of homotopy fibers of
the vertical maps. Considering the given map $B\to Y$ as a point
in $\map_{\mathcal{C}_{A/}}\left(B,Y\right)$ and considering the
induced equivalence on the homotopy fibers of the vertical maps, we
obtain by T.5.5.5.12 an equivalence
\[
\map_{\mathcal{C}_{A//Y}}\left(\overline{B},\overline{X}\right)\iso\map_{\mathcal{C}_{A//W}}\left(\overline{\overline{B}},\overline{\overline{Z}}\right),
\]
where $\overline{B}$ and $\overline{X}$ are $A\to B\to Y$ and $A\to X\to Y$
viewed as objects of $\mathcal{C}_{A//Y}$ and $\overline{\overline{B}}$
and $\overline{\overline{Z}}$ are $A\to B\to W$ and $A\to Z\to W$
viewed as objects of $\mathcal{C}_{A//W}$. By the definition of the
space of lifts, this is precisely the equivalence $L\left(q_{l}\right)\simeq L\left(q\right)$.
\end{proof}

The following lemma expands on remark T.5.2.8.7:
\begin{lem}
\label{lem:Lifts_Adjoints} Let $F\colon \mathcal{C}\adj\mathcal{D}\noloc G$
be an adjunction of $\infty$-categories. For every commutative square
$q\colon \Delta^{1}\times\Delta^{1}\to\mathcal{D}$ of the form
\[
\xymatrix{F\left(A\right)\ar[d]^{F\left(f\right)}\ar[r] & X\ar[d]^{g}\\
F\left(B\right)\ar[r] & Y
,}
\]
there is an adjoint square $p\colon \Delta^{1}\times\Delta^{1}\to\mathcal{C}$
of the form
\[
\xymatrix{A\ar[d]^{f}\ar[r] & G\left(X\right)\ar[d]^{G\left(g\right)}\\
B\ar[r] & G\left(Y\right)
}
\]
and a canonical homotopy equivalence $L\left(q\right)\simeq L\left(p\right)$.
\end{lem}

\begin{proof}
Let $\mathcal{M}\to\Delta^{1}$ be the Cartesian-coCartesian fibration
associated with the adjunction $F\dashv G$. Since $\mathcal{C}$
and $\mathcal{D}$ are full subcategories of $\mathcal{M}$ we can
think of the square $q$ as taking values in $\mathcal{M}$ and it
does not change the space of lifts. Consider the diagram in $\mathcal{M}$
given by
\[
\xymatrix{A\ar[r]\ar[d]^{f} & F\left(A\right)\ar[d]^{F\left(f\right)}\ar[r] & X\ar[d]^{g}\\
B\ar[r] & F\left(B\right)\ar[r] & Y
,}
\]
where in the left square $q_{l}$ the horizontal arrows are coCartesian
and the rest of the data is given by the lifting property of coCartesian
edges. Since the inclusion of the spine $\Lambda_{1}^{2}\into\Delta^{2}$
is inner anodyne, so is $\Delta^{1}\times\Lambda_{1}^{2}\into\Delta^{1}\times\Delta^{2}$
(by T.2.3.2.4) and since $\mathcal{M}\to\Delta^{1}$ is an inner fibration,
the diagram can be extended to $\Delta^{1}\times\Delta^{2}\to\mathcal{M}$
and we can denote the outer square by $r\colon \Delta^{1}\times\Delta^{1}\to\mathcal{M}$. We now claim that $q_{l}$ is a pushout square in
$\mathcal{M}$. For every $Z\in\mathcal{M}$, consider the induced
diagram
\[
\xymatrix{\map\left(F\left(B\right),Z\right)\ar[d]\ar[r] & \map\left(B,Z\right)\ar[d]\\
\map\left(F\left(A\right),Z\right)\ar[r] & \map\left(A,Z\right)
.}
\]
If $Z\in\mathcal{M}_{0}\simeq\mathcal{C}$, then the spaces on both
left corners are empty and if $Z\in\mathcal{M}_{1}\simeq\mathcal{D}$, then both horizontal arrows are equivalences. Either way, this is
a pullback square and hence $q_{l}$ is a pushout square. By \lemref{Lifts_Pullback_Pushout}
we get $L\left(q\right)\simeq L\left(r\right)$. 

We can now factor the outer square $r\colon \Delta^{1}\times\Delta^{1}\to\mathcal{M}$
as 
\[
\xymatrix{A\ar[r]\ar[d]^{f} & G\left(X\right)\ar[d]^{G\left(g\right)}\ar[r] & X\ar[d]^{g}\\
B\ar[r] & G\left(Y\right)\ar[r] & Y
,}
\]
where the left square is $p$ and in the right square $q_{r}$ the
horizontal arrows are Cartesian and the square is determined by the
lifting property of Cartesian edges. Repeating the argument in the
dual form we get that $q_{r}$ is a pullback square and using \lemref{Lifts_Pullback_Pushout}
again we get $L\left(p\right)\simeq L\left(r\right)$ and therefore
$L\left(p\right)\simeq L\left(q\right)$.
\end{proof}
\begin{prop}
\label{prop:Lifts_Diagonal} Let $\mathcal{C}$ be an $\infty$-category.
Let $q\colon \Delta^{1}\times\Delta^{1}\to\mathcal{C}$ be a commutative
square
\[
\xymatrix{A\ar[d]_{f}\ar[r]^{\alpha} & X\ar[d]^{g}\\
B\ar[r]^{\beta} & Y
,}
\]
with space of lifts $L\left(q\right)$. Given a point $\left(s_{0},s_{1}\right)\in L\left(q\right)\times L\left(q\right)$,
the homotopy fiber of the diagonal 
\[
\delta_{L\left(q\right)}\colon L\left(q\right)\to L\left(q\right)\times L\left(q\right)
\]
over $\left(s_{0},s_{1}\right)$ is homotopy equivalent to the space
of lifts for a square $p\colon \Delta^{1}\times\Delta^{1}\to\mathcal{C}$
of the form
\[
\xymatrix{A\ar[d]_{f}\ar[r]^{\alpha} & X\ar[d]^{\delta g}\\
B\ar[r]^{\left(s_{0},s_{1}\right)\quad} & X\times_{Y}X
.}
\]
\end{prop}

\begin{proof}
For ease of notation, set $\mathcal{D}=\mathcal{C}_{A/}$. Recall
that
\[
L\left(q\right)=\map_{\mathcal{D}_{/\overline{Y}}}\left(\overline{B},\overline{X}\right)
\]
and therefore 
\[
L\left(q\right)\times L\left(q\right)=\map_{\mathcal{D}_{/\overline{Y}}}\left(\overline{B},\overline{X}\right)\times\map_{\mathcal{D}_{/\overline{Y}}}\left(\overline{B},\overline{X}\right)\simeq\map_{\mathcal{D}_{/\overline{Y}}}\left(\overline{B},\overline{X}\times\overline{X}\right).
\]
Products in the over-category are fibered products and products in
the under-category are just ordinary products (dual of T.1.2.13.8).
Hence, $\overline{X}\times\overline{X}$ is the diagram $A\to X\times_{Y}X\to Y$, which we denote by $\overline{X\times_{Y}X}$. Thus, a point $s=\left(s_{0},s_{1}\right)\in L\left(q\right)\times L\left(q\right)$
corresponds to a lift in the diagram
\[
\xymatrix{ & \overline{X\times_{Y}X}\ar[d]\\
\overline{B}\ar[r]\ar@{-->}[ru] & \overline{Y}
}
\]
in the category $\mathcal{D}$. Furthermore, the diagonal map $\delta_{L\left(q\right)}\colon L\left(q\right)\to L\left(q\right)\times L\left(q\right)$
is induced from the diagonal map $\delta_{\overline{X}}\colon \overline{X}\to\overline{X}\times\overline{X}$.
Namely, $\delta_{L\left(q\right)}=\left(\delta_{\overline{X}}\right)_{*}$.
Our goal is therefore to compute the homotopy fiber of $\left(\delta_{\overline{X}}\right)_{*}$
over a given point 
\[
s=\left(s_{0},s_{1}\right)\simeq\map_{\mathcal{D}_{/\overline{Y}}}\left(\overline{B},\overline{X}\times\overline{X}\right).
\]
The projection $\mathcal{D}_{/\overline{Y}}\to\mathcal{D}$ induces
an equivalence
\[
\left(\mathcal{D}_{/\overline{Y}}\right)_{/\overline{X\times_{Y}X}}\simeq\mathcal{D}_{/\overline{X\times_{Y}X}}.
\]
It follows that the fiber is the space of lifts in the diagram
\[
\xymatrix{ & \overline{X}\ar[d]\\
\overline{B}\ar[r]^{s} \ar@{-->}[ru] & \overline{X\times_{Y}X}
}
\]
in $\mathcal{D}$. By (the dual of) T.5.5.5.12, this space of lifts is homotopy equivalent to the mapping space 
$\map_{\mathcal{D}_{/\overline{X\times_{Y}X}}}\left(\overline{B},\overline{X}\right)$.
Recalling that $\mathcal{D}=\mathcal{C}_{A/}$, we see that this is
none other than the space of lifts for $p$.
\end{proof}

\subsection{Truncatedness and Connectedness}

We recall the following definition from classical homotopy theory:
\begin{defn}
For $d\ge-2$, a map $f\colon X\to Y$ of spaces is called\emph{ $d$-truncated} if all of its homotopy fibers are $d$-truncated spaces (\defref{Truncated_Space}).
\end{defn}

Using this definition, one can define a general notion of $d$-truncatedness in
an $\infty$-category.

\begin{defn}(T.5.5.6.1) 
For $d\ge-2$, a map $f\colon X\to Y$ in an $\infty$-category
$\mathcal{C}$ is called \emph{$d$-truncated}, if for every $Z\in\mathcal{C}$
the induced map
\[
\map\left(Z,X\right)\to\map\left(Z,Y\right)
\]
is a $d$-truncated map of spaces. An object $X$ is $d$-truncated,
if the map $X\to\term_{\mathcal{C}}$ is $d$-truncated. We denote
by $\tau_{\le d}\mathcal{C}$ the full subcategory of $\mathcal{C}$
spanned by the $d$-truncated objects. When $\mathcal{C}$ is presentable,
by T.5.5.6.21 the $\infty$-category $\tau_{\le d}\mathcal{C}$ is
itself presentable and by T.5.5.6.18, the inclusion $\tau_{\le d}\mathcal{C}\into\mathcal{C}$
has a left adjoint $\tau_{\le d}^{\mathcal{C}}\colon \mathcal{C}\to\tau_{\le d}\mathcal{C}$. 
\end{defn}

\begin{rem}
It is not difficult to show that $\tau_{\le d}$ extends to a functor
from the $\infty$-category of presentable $\infty$-categories to the
full subcategory spanned by presentable essentially $\left(d+1\right)$-categories
and that it is left adjoint to the inclusion. The maps $\tau_{\le d}^{\mathcal{C}}$
can be taken to be the components of the unit transformation (this
essentially follows from T.5.5.6.22), but we shall not need this.
\end{rem}

We now turn to discuss the dual notion of $n$-connectedness. 
\begin{defn}
For $n\ge-2$, a map $f\colon A\to B$ in an $\infty$-category $\mathcal{C}$
is $n$\emph{-connected} if it is left orthogonal to every $n$-truncated
map; ie for every commutative square $q\colon \Delta^{1}\times\Delta^{1}\to\mathcal{C}$,
\[
\xymatrix{A\ar[d]_{f}\ar[r] & X\ar[d]^{g}\\
B\ar[r] & Y
,}
\]
in which $g$ is $n$-truncated, $L\left(q\right)$ is contractible.
An object $A\in\mathcal{C}$ is called $n$-connected if $A\to\term_{\mathcal{C}}$
is $n$-connected.
\end{defn}

\begin{lem}
\label{lem:Connectedness_Adjoints}Let $\mathcal{C}$ and $\mathcal{D}$
be $\infty$-categories that admit finite limits and let
$F\colon \mathcal{C}\adj\mathcal{D}\noloc G$ be an adjunction with $F\dashv G$,

\begin{lemenum}
\item For every $d\ge-2$ and a $d$-truncated morphism $g$ in $\mathcal{D}$,
the morphism $G\left(g\right)$ is a $d$-truncated morphism in $\mathcal{C}$.
\item For every $n\ge-2$ and an $n$-connected morphism $f$ in $\mathcal{C}$,
the morphism $F\left(f\right)$ is an $n$-connected morphism in $\mathcal{D}$.
\end{lemenum}
\end{lem}

\begin{proof}
As a right adjoint, $G$ is left exact and therefore preserves $d$-truncated
morphisms by T.5.5.6.16. Since $G$ preserves $n$-truncated morphisms
and the space of lifts in the square 
\[
\Square{F\left(A\right)}{F\left(B\right)}XY
\]
is homotopy equivalent to the space of lifts in the adjoint square
\[
\Square AB{G\left(X\right)}{G\left(Y\right)}
\]
given by \lemref{Lifts_Adjoints}, we see that if $f$ is left orthogonal to all $n$-truncated morphisms
then so is $F\left(f\right)$.
\end{proof}
\begin{lem}
\label{lem:connected_via_truncation} Let $\mathcal{C}$ be a presentable $\infty$-category, let $f\colon A\to B$ be a morphism in $\mathcal{C}$, and let $n\ge-2$ be an integer. The map $f$ is $n$-connected if and only if viewed as an object $\overline{A}$ of $\mathcal{C}_{/B}$, its $n$-truncation $\tau_{\le n}^{\mathcal{C}_{/B}}\left(\overline{A}\right)$ is the terminal object (ie $\Id_{B}\colon B\to B$).
\end{lem}

\begin{proof}
Since $\mathcal{C}$ has all pullbacks, every commutative square $q\colon \Delta^{1}\times\Delta^{1}\to\mathcal{C}$
of the form
\[
\xymatrix{A\ar[d]_{f}\ar[r] & X\ar[d]\\
B\ar[r] & Y
}
\]
can be factored as 
\[
\Rect AB{B\times_{Y}X}BX{Y.}
\]
By \lemref{Lifts_Pullback_Pushout}, the space of lifts for the original
square $q$ is equivalent to the space of lifts in the left square
of the above rectangle. Moreover, $n$-truncated morphisms are closed
under base change and so to check that $f$ is $n$-connected, we can
equivalently restrict ourselves to checking the left orthogonality
condition only for squares $q$ in which the map $B\to Y$ is the
identity on $B$. Writing $\overline{A}$, $\overline{X}$ and $\overline{B}$
for $A\to B$, $X\to B$ and $\Id\colon B\to B$ as objects of $\mathcal{C}_{/B}$, respectively, we see that by the dual of T.5.5.5.12 the space of lifts fits into
a fiber sequence 
\[
L\left(q\right)=\map_{\mathcal{C}_{A//B}}\left(\overline{B},\overline{X}\right)\to\map_{\mathcal{C}_{/B}}\left(\overline{B},\overline{X}\right)\xrightarrow{f^{*}}\map_{\mathcal{C}_{/B}}\left(\overline{A},\overline{X}\right).
\]
Hence, $f$ is $n$-connected if and only if $f^{*}$ is an equivalence
for every $n$-truncated morphism $X\to B$. By T.5.5.6.10, a morphism
$X\to B$ is $n$-truncated if and only if $\overline{X}$ is an $n$-truncated
object of $\mathcal{C}_{/B}$. Hence, we need the above map to be
an equivalence for every $n$-truncated object $\overline{X}\in\mathcal{C}_{/B}$.
This precisely means that the map $\overline{A}\to\overline{B}$ exhibits
$\overline{B}$, the terminal object of $\mathcal{C}_{/B}$, as the
$n$-truncation of $\overline{A}$.
\end{proof}
\begin{cor}
In a presentable $\infty$-category $\mathcal{C}$, an object $X$
is $n$-connected for some $n\ge-2$ if and only if its $n$-truncation
$\tau_{\le n}^{\mathcal{C}}X$ is a terminal object of $\mathcal{C}$.
\end{cor}

The following is a quantitative generalization of the defining property
of an $n$-connected morphism.
\begin{prop}
\label{prop:Lifts_Space_Truncatedness} Let $\mathcal{C}$ be a presentable
$\infty$-category. Fix integers $d\ge n\ge-2$. For every square
$q\colon \Delta^{1}\times\Delta^{1}\to\mathcal{C}$ of the form
\[
\Square ABX{Y,}
\]
in which $f\colon A\to B$ is $n$-connected and $g\colon X\to Y$ is $d$-truncated,
the space of lifts $L\left(q\right)$ is $\left(d-n-2\right)$-truncated.
\end{prop}

\begin{proof}
We prove this by induction on $d$. For $d=n$, the claim follows from the
definition of an $n$-connected morphism and the fact that a space is $\left(-2\right)$-connected if and only if it is contractible. We now assume that this is true for $d-1$, and prove it for $d$. Denote the space of lifts by $L\left(q\right)$. 
By T.5.5.6.15, it suffices to show that the diagonal map $\delta\colon L\left(q\right)\to L\left(q\right)\times L\left(q\right)$
is $\left(d-n-3\right)$-truncated. By \propref{Lifts_Diagonal},
the homotopy fiber over a point $\left(s_{0},s_{1}\right)\in L\left(q\right)\times L\left(q\right)$
is equivalent to the space of lifts in the square
\[
\LP ABX{X\times_{Y}X,}
\]
where the bottom map is $\left(s_{0},s_{1}\right)$. By T.5.5.6.15,
since $X\to Y$ is $d$-truncated, $X\to X\times_{Y}X$ is $\left(d-1\right)$-truncated
and, therefore, by induction, the space of lifts is $\left(\left(d-1\right)-n-2\right)$-truncated and we are done.
\end{proof}

\subsection{$(n-\frac{1}{2})$-connectedness}

We begin by introducing an auxiliary notion that will be helpful in the study of $n$-connectedness.
\begin{defn}
For every $n\ge-2$, a morphism $f\colon X\to Y$ is called \emph{$\left(n-\frac{1}{2}\right)$-connected}
if the induced map $\tau_{\le n}^{\mathcal{C}}\left(f\right)\colon \tau_{\le n}^{\mathcal{C}}X\to\tau_{\le n}^{\mathcal{C}}Y$
is an equivalence.
\end{defn}

To justify the terminology we need to show that it indeed sits between
$n$ and $\left(n-1\right)$-connectedness, at least under some reasonable
conditions. One direction is completely general:
\begin{lem}
Let $n\ge-2$ and let $\mathcal{C}$ be a presentable $\infty$-category.
If a morphism $f\colon A\to B$ is $n$-connected, then it is $\left(n-\frac{1}{2}\right)$-connected. 
\end{lem}

\begin{proof}
By the Yoneda lemma it is enough to show that for every $n$-truncated
object $Z$ in $\mathcal{C}$ the induced map 
\[
f_{*}\colon \map\left(B,Z\right)\to\map\left(A,Z\right)
\]
is an equivalence. For this, it is enough to show that for every $g\colon A\to Z$,
the fiber of $f_{*}$ over $g$ is contractible. By T.5.5.5.12, the
fiber is equivalent to the space of lifts for the square
\[
\Square ABZ{\term,}
\]
which is contractible by definition as $f\colon A\to B$ was assumed to be $n$-connected.
\end{proof}
For the other direction, we need to assume that our $\infty$-category
is an $m$-topos. First,
\begin{lem}
\label{lem:topos_truncation_pullback}Let $\mathcal{C}$ be an $m$-topos
for some $-1\le m\le\infty$. For every $d$-truncated morphism $g\colon X\to Y$,
the diagram 
\[
\Square XY{\tau_{\le d+1}^{\mathcal{C}}X}{\tau_{\le d+1}^{\mathcal{C}}Y}
\]
is a pullback square.
\end{lem}

\begin{proof}
For $\mathcal{C}=\mathcal{S}$, this follows from inspecting the induced
map between the long exact sequences of homotopy groups associated
with the vertical maps. For $\mathcal{C}=\mathcal{S}^{K}$, this follows
from the claim for $\mathcal{S}$, since both truncation and pullbacks
are computed level-wise. A general $\infty$-topos is a left exact
localization of $\mathcal{S}^{K}$ for some $K$, and left exact colimit-preserving functors between presentable $\infty$-categories commute
with truncation by T.5.5.6.28 and with pullbacks by assumption. Finally,
by T.6.4.1.5 every $m$-topos is the full subcategory on $\left(m-1\right)$-truncated
objects in an $\infty$-topos and this full subcategory is closed
under limits.
\end{proof}
From this we deduce
\begin{lem}
\label{lem:n.5_to_n_connected}Let $n\ge-2$ and let $\mathcal{C}$
be an $m$-topos for some $-1\le m\le\infty$. If a morphism $f\colon A\to B$
is $\left(n+\frac{1}{2}\right)$-connected then it is $n$-connected.
\end{lem}

\begin{proof}
To show that $f\colon A\to B$ is $n$-connected, we need to show that the
space of lifts for every square
\[
\Square ABX{Y,}
\]
in which the right vertical arrow is $n$-truncated, is contractible.
Applying \lemref{topos_truncation_pullback} and \lemref{Lifts_Pullback_Pushout},
we see that this space is equivalent to the space of lifts in the square
\[
\Square AB{\tau_{\le n+1}^{\mathcal{C}}X}{\tau_{\le n+1}^{\mathcal{C}}Y,}
\]
which, by \lemref{Lifts_Adjoints}, is equivalent to the space of lifts
in the adjoint square
\[
\Square{\tau_{\le n+1}^{\mathcal{C}}A}{\tau_{\le n+1}^{\mathcal{C}}B}{\tau_{\le n+1}^{\mathcal{C}}X}{\tau_{\le n+1}^{\mathcal{C}}Y,}
\]
which is contractible since the left vertical arrow is an equivalence.
\end{proof}
As a consequence, we obtain another sense in which $\left(n-\frac{1}{2}\right)$-connected
morphisms are ``close'' to being $n$-connected:
\begin{prop}
\label{prop:Connectedness_Section}Let $n\ge-2$ and let $\mathcal{C}$
be an $m$-topos for some $-1\le m\le\infty$. If a morphism $f\colon A\to B$
in $\mathcal{C}$ is $\left(n-\frac{1}{2}\right)$-connected and has
a section (ie there exists $s\colon B\to A$ such that $f\circ s\sim \Id_{B}$), then $f$ is $n$-connected.
\end{prop}

\begin{proof}
We first prove the case of $m=\infty$. For $n=-2$, there is nothing
to prove, and so we assume that $n\ge-1$. Since $f\circ s=\Id_{B}$ we get $\tau_{\le n}^{\mathcal{C}}\left(f\right)\circ\tau_{\le n}^{\mathcal{C}}\left(s\right)=\Id_{B}$
and since $\tau_{\le n}\left(f\right)$ is an equivalence, then so
is $\tau_{\le n}\left(s\right)$ and hence $s$ is $\left(n-\frac{1}{2}\right)$-connected.
By \lemref{n.5_to_n_connected}, $s$ is $\left(n-1\right)$-connected
and hence, by T.6.5.1.20, the map $f$ is $n$-connected (note that
$n$-\emph{connective} means $\left(n-1\right)$-connected).

For a general $m$, by T.6.4.1.5 there exists an $\infty$-topos $\mathcal{D}$
and an equivalence $\mathcal{C}\simeq\tau_{\le m-1}\mathcal{D}$, and so
we may identify $\mathcal{C}$ with the full subcategory of $\left(m-1\right)$-truncated
objects of $\mathcal{D}$. If $f\colon A\to B$ is $\left(n-\frac{1}{2}\right)$-connected
in $\mathcal{C}$, then it is also $\left(n-\frac{1}{2}\right)$-connected
in $\mathcal{D}$, since the restriction of $\tau_{\le n}^{\mathcal{D}}$
to $\mathcal{C}$ is equivalent to $\tau_{\le n}^{\mathcal{C}}$.
It follows from the case of $m=\infty$ that
$f$ is $n$-connected in $\mathcal{D}$. Since $f=\tau_{\le m-1}^{\mathcal{D}}f$ and $\tau_{\le m-1}^{\mathcal{D}}$ is a left adjoint functor, by
\lemref{Connectedness_Adjoints} the map $f$ is also $n$-connected
as a map in $\mathcal{C}$.
\end{proof}

\subsection{Connectedness in Algebras}

We begin with the following general fact:
\begin{lem}
\label{lem:Monadic_Connectedness}Let $F\colon \mathcal{C}\adj\mathcal{D}\noloc U$
be a monadic adjunction between presentable $\infty$-categories.
If the monad $T=U\circ F$ preserves $n$-connected morphisms, then
$U$ detects $n$-connected morphisms. Namely, given a morphism $f\colon A\to B$
in $\mathcal{D}$, if $U\left(f\right)$ is $n$-connected for some
$n\ge-2$, then $f$ is $n$-connected.
\end{lem}

\begin{proof}
Given a morphism $f\colon A\to B$ in $\mathcal{D}$, using the canonical
simplicial resolution provided by the proof of A.4.7.3.13, we can
express it as a colimit of the simplicial diagram of morphisms:

\[
\colim\limits _{\Delta^{op}}\left(T^{n+1}\left(A\right)\to T^{n+1}\left(B\right)\right),
\]
which one can write as
\[
\colim\limits _{\Delta^{op}}\left(FT^{n}U\left(A\right)\to FT^{n}U\left(B\right)\right).
\]
If $U\left(f\right)$ is $n$-connected as in the statement, then
since $T$ preserves $n$-connected morphisms by assumption and $F$
preserves $n$-connected morphisms by being left adjoint, it follows
that all the maps in the diagram are $n$-connected. By T.5.2.8.6(7),
the map $f$ is also $n$-connected.
\end{proof}
We want to apply the above to the free-forgetful adjunction between
a symmetric monoidal $\infty$-category $\mathcal{C}$ and the category
of $\mathcal{P}$-algebras in $\mathcal{C}$, where $\mathcal{P}$
is a reduced $\infty$-operad. For this, we need some compatibility between
the notion of $n$-connectedness and the symmetric monoidal structure: 
\begin{lem}
\label{lem:Connectedness_Tensor_Closure}Let $\mathcal{C}$ be a presentably
symmetric monoidal $\infty$-category. For every integer $n\ge-2$,
the class of $n$-connected morphisms in $\mathcal{C}$ is closed
under tensor products.
\end{lem}

\begin{proof}
Since $\mathcal{C}$ is presentable and the tensor product commutes
with colimits separately in each variable, for each object $X\in\mathcal{C}$
the functor $Y\mapsto X\otimes Y$ is a left adjoint and therefore
preserves $n$-connected morphisms by \lemref{Lifts_Adjoints}. Hence,
given two $n$-connected morphisms $f\colon A_{1}\to B_{1}$ and $g\colon A_{2}\to B_{2}$,
the composition 
\[
A_{1}\otimes B_{1}\xrightarrow{A_{1}\otimes g}A_{1}\otimes B_{2}\xrightarrow{f\otimes B_{2}}A_{2}\otimes B_{2}
\]
is $n$-connected as a composition of two $n$-connected morphisms.
\end{proof}
\begin{example}
\label{Topos_Cartesian} For every $m$-topos (with $-1\le m\le\infty$)
and $n\ge-2$, the class of $n$-connected morphisms is closed under
Cartesian products. In particular, this applies to $\mathcal{S}_{\le m}^{K}$
for every simplicial set $K$.
\end{example}

\begin{lem}
\label{lem:Algebraic_Monad}Let $\mathcal{P}$ be a reduced $\infty$-operad
and let $\mathcal{C}$ be a presentably symmetric monoidal $\infty$-category.
The free-forgetful adjunction 
\[
F\colon \mathcal{C}\adj\alg_{\mathcal{P}}\left(\mathcal{C}\right)\noloc U
\]
is monadic and the associated monad $T=U\circ F$ preserves $n$-connected
morphisms.
\end{lem}

\begin{proof}
By A.4.7.3.11, the adjunction $F\dashv U$ is monadic. Hence, given
a morphism $A\to B$ in $\mathcal{C}$, by \propref{Monad_Map_Formula}
the morphism $T\left(A\right)\to T\left(B\right)$ can be expressed
as
\[
\coprod_{n\ge0}\left(P\left(n\right)\otimes A^{\otimes n}\right)_{h\Sigma_{n}}\to\coprod_{n\ge0}\left(P\left(n\right)\otimes B^{\otimes n}\right)_{h\Sigma_{n}}.
\]

By \lemref{Connectedness_Tensor_Closure}, $n$-connected morphisms
are closed under $\otimes$ and, by T.5.2.8.6, they are closed under colimits. Hence, we obtain that $T\left(A\right)\to T\left(B\right)$
is $n$-connected as well. 
\end{proof}

\begin{prop}
\label{prop:Algebraic_Reduction}Let $\mathcal{P}$ be a reduced $\infty$-operad
and let $\mathcal{C}$ be a presentably symmetric monoidal $\infty$-category.
Given a morphism $f\colon A\to B$ in $\alg_{\mathcal{P}}\left(\mathcal{C}\right)$,
if the underlying map $U\left(f\right)$ is $n$-connected for some
$n\ge-2$, then $f$ is $n$-connected.
\end{prop}

\begin{proof}
Combine \lemref{Algebraic_Monad} and \lemref{Monadic_Connectedness}.
\end{proof}

\section{The $\infty$-Categorical Eckmann\textendash Hilton Argument}

In this final section we prove our main results. In 5.1 we analyze
the canonical map from the coproduct to the tensor product of two
algebras over a reduced $\infty$-operad. The main result is that
under suitable assumptions, if the $\infty$-operad is highly connected,
then this map is also highly connected (\propref{Coproduct_Lemma}).
In 5.2 we use the connectivity bound established in 5.1 to analyze
the reduced endomorphism operad of an object in an $\infty$-topos.
This analysis recovers and expands on classical results on deloopings of spaces with non-vanishing homotopy groups in a bounded region. 
In 5.3 we prove our main theorem (\thmref{Main_Theorem}) and its main corollary: the $\infty$-categorical Eckmann\textendash Hilton argument (\corref{EHA}).
We conclude with some curious applications of the main theorem to some
questions regarding tensor products of reduced $\infty$-operads.

\subsection{Coproducts of Algebras}

Let $\mathcal{C}$ be a symmetric monoidal $\infty$-category and
let $\mathcal{P}$ be a reduced $\infty$-operad. For every two algebras
$A,B\in\alg_{\mathcal{P}}\left(\mathcal{C}\right)$, there is a canonical
map of algebras 
\[
f_{A,B}\colon A\sqcup B\to A\otimes B
\]
formally given by 
\[
f_{A,B}=\Id_{A}\otimes1_{B}\sqcup1_{A}\otimes \Id_{B},
\]
where $1_{A}\colon 1\to A$ and $1_{B}\colon 1\to B$ are the respective unit
maps viewed as maps of algebras (see A.3.2.1). 
\begin{lem}
\label{lem:coproduct_tensor_section}Let $\mathcal{C}$ be a symmetric
monoidal $\infty$-category and let $\mathcal{P}$ be a reduced $\infty$-operad.
If $\mathcal{P}\not\simeq\bb E_{0}$, then for every pair of algebras
$A,B\in\alg_{\mathcal{P}}\left(\mathcal{C}\right)$, the canonical
map
\[
f_{A,B}\colon A\sqcup B\to A\otimes B
\]
has a section after we apply the forgetful functor $\underline{\left(-\right)}\colon \alg_{\mathcal{P}}\left(\mathcal{C}\right)\to\mathcal{C}$.
\end{lem}

\begin{proof}
By \remref{Minus_One_Connected}, if $\mathcal{P}\not\simeq\bb E_{0}$,
then it is $\left(-1\right)$-connected and in particular $\mathcal{P}\left(2\right)\neq\es$.
We shall construct a section to $\underline{f_{A,B}}$ using any binary
operation $\mu\in\mathcal{P}\left(2\right)$. Let $i_{A}\colon A\to A\sqcup B$
and $i_{B}\colon B\to A\sqcup B$ be the canonical maps of the coproduct.
Define $s$ to be the composition of the following maps: 
\[
\underline{A}\otimes\underline{B}\xrightarrow{\underline{i_{A}}\otimes\underline{i_{B}}}\left(\underline{A\sqcup B}\right)\otimes\left(\underline{A\sqcup B}\right)\xrightarrow{\mu_{A\sqcup B}}\underline{A\sqcup B}.
\]
Now, consider the following diagram in the homotopy category of $\mathcal{C}$:
\[
\xymatrix{ &  &  &  & \left(\underline{A\sqcup B}\right)\otimes\left(\underline{A\sqcup B}\right)\ar[d]^{\underline{f_{A,b}}\otimes\underline{f_{A,B}}}\ar[rrr]^{\mu_{A\sqcup B}} &  &  & \underline{A\sqcup B}\ar[d]^{\underline{f_{A,B}}}\\
\underline{A\otimes B}\ar[rrrru]^{\underline{i_{A}}\otimes\underline{i_{B}}}\ar[rrrr]^{\underline{\left(\Id_{A}\otimes1_{B}\right)\otimes\left(1_{A}\otimes \Id_{B}\right)}}\ar[rrrrd]_{\underline{\left(\Id_{A}\otimes1_{A}\right)\otimes\left(1_{B}\otimes \Id_{B}\right)}} &  &  &  & \left(\underline{A\otimes B}\right)\otimes\left(\underline{A\otimes B}\right)\ar[rrr]^{\mu_{A\otimes B}}\ar[d]^{\Id_{\underline{A}}\otimes\sigma_{\underline{A},\underline{B}}\otimes \Id_{\underline{B}}} &  &  & \underline{A\otimes B}\ar@{=}[d]\\
 &  &  &  & \left(\underline{A\otimes A}\right)\otimes\left(\underline{B\otimes B}\right)\ar[rrr]^{\mu_{A}\otimes \mu_{B}} &  &  & \underline{A\otimes B}
}
.
\]
The upper square commutes since $f_{A,B}$ is a map of algebras. The
upper triangle commutes since it is the tensor product of two triangles,
which commute by the very definition of $f_{A,B}$. The lower square
commutes by the definition of the algebra structure on $A\otimes B$
and the lower triangle also clearly commutes. The composition of the
bottom diagonal map and the bottom right map is the identity, since
the restriction of $\mu$ to the unit in one of the arguments is homotopic to
the identity map of the other argument. The composition of the top
diagonal map with the top right map is $s$. It follows that $\underline{f_{A,B}}\circ s\sim \Id_{\underline{A\otimes B}}$.
\end{proof}
\begin{lem}
\label{lem:Doctrinal_adjunction} Let $\mathcal{C}$ and $\mathcal{D}$
be symmetric monoidal $\infty$-categories and let $F\colon \mathcal{C}\to\mathcal{D}$
be a symmetric monoidal functor. If $\underline{F}\colon \underline{\mathcal{C}}\to\underline{\mathcal{D}}$
is a left adjoint, then the induced functor $F^{\otimes}\colon \mathcal{C}^{\otimes}\to\mathcal{D}^{\otimes}$
is a left adjoint relative to $\finpt$ and for every $\infty$-operad
$\mathcal{P}$ the induced functor $\underline{\alg}_{\mathcal{P}}\left(\mathcal{C}\right)\to\underline{\alg}_{\mathcal{P}}\left(\mathcal{D}\right)$
is a left adjoint.
\end{lem}

\begin{proof}
For every $\left\langle n\right\rangle \in\finpt$, the restriction
of $F^{\otimes}$ to the fiber over $\left\langle n\right\rangle $
is just $F^{n}\colon \mathcal{C}^{n}\to\mathcal{D}^{n}$, which is clearly
a left adjoint. Hence, by A.7.3.2.7, the functor $F^{\otimes}$ is
a left adjoint relative to $\finpt$. Let $G^{\otimes}$ be the right
adjoint of $F^{\otimes}.$ Applying A.7.3.2.13, we obtain that $F^{\otimes}$
and $G^{\otimes}$ induce an adjunction: 
\[
\underline{\alg}_{\mathcal{P}}\left(\mathcal{C}\right)\adj\underline{\alg}_{\mathcal{P}}\left(\mathcal{D}\right).
\]
\end{proof}
In what follows we are going to restrict ourselves to the case of
a Cartesian monoidal structure. The next proposition is the key connectivity bound on which the main theorems of this paper rest.
\begin{prop}
\label{prop:Coproduct_Lemma} Let $\mathcal{C}$ be an $m$-topos
for some $-1\le m\le\infty$ with the Cartesian symmetric monoidal
structure and let $\mathcal{P}$ be a reduced $d$-connected $\infty$-operad
for some $d\ge-2$. For every pair of algebras $A,B\in\alg_{\mathcal{P}}\left(\mathcal{C}\right)$,
the canonical map
\[
f_{A,B}\colon A\sqcup B\to A\times B
\]
is $d$-connected.
\end{prop}

\begin{proof}
For $d=-2$ there is nothing to prove and so we assume that $d\ge-1$.
By \propref{Algebraic_Reduction}, it is enough to show that $\underline{f_{A,B}}$
is $d$-connected where $\underline{\left(-\right)}\colon \alg_{\mathcal{P}}\left(\mathcal{C}\right)\to\mathcal{C}$
is the forgetful functor. By \propref{Connectedness_Section}, it is
enough to show that $\underline{f_{A,B}}$ has a section and is $\left(d-\frac{1}{2}\right)$-connected.
Since $d\ge-1$, we have $\mathcal{P} \neq \mathbb{E}_0$ and, therefore,
by \lemref{coproduct_tensor_section}, $\underline{f_{A,B}}$ has
a section. Thus, we are reduced to showing that the image of $\underline{f_{A,B}}$
under the functor $\tau_{\le d}^{\mathcal{C}}\colon \mathcal{C}\to\tau_{\le d}\mathcal{C}$
is an equivalence. First, we show that $\tau_{\le d}^{\mathcal{C}}$
preserves binary products. For $m=\infty$, this follows from T.6.5.1.2.
The general case reduces to $m=\infty$ as by T.6.4.1.5 we can embed
$\mathcal{C}$ as a full subcategory of an $\infty$-topos spanned
by the $\left(m-1\right)$-truncated objects. It follows that we get
a symmetric monoidal functor $\tau_{\le d}^{\times}\colon \mathcal{C}^{\times}\to\left(\tau_{\le d}\mathcal{C}\right)^{\times}$.
By \lemref{Doctrinal_adjunction}, the functor
\[
F\colon \alg_{\mathcal{P}}\left(\mathcal{C}\right)\to\alg_{\mathcal{P}}\left(\tau_{\le d}\mathcal{C}\right)
\]
induced by $\tau_{\le n}^{\times}$ is a left adjoint. Consider the
following (solid) commutative diagram in the homotopy category of
$\cat$:

\[
\xymatrix{\alg_{\mathcal{P}}\left(\mathcal{C}\right)\ar[r]^{F}\ar[d] & \alg_{\mathcal{P}}\left(\tau_{\le d}\mathcal{C}\right)\ar[d]\ar@<1pt>@{-->}[r]^{G'} & \alg_{\bb E_{\infty}}\left(\tau_{\le d}\mathcal{C}\right)\ar@<1pt>[l]^{G}\ar[d]\\
\mathcal{C}\ar[r]^{\tau_{\le d}^{\mathcal{C}}} & \tau_{\le d}\mathcal{C} & \tau_{\le d}\mathcal{C},\ar@{=}[l]
}
\]
where the vertical maps are the forgetful functors and $G$ is induced
by restriction along the essentially unique map $\mathcal{P}\to\bb E_{\infty}$. Since $\tau_{\le d}\mathcal{C}$
is an essentially $\left(d+1\right)$-category, it follows from \propref{d-homotopy_operad}
that $G$ is an equivalence. Taking $G'$ to be an inverse of $G$
up to homotopy, the outer rectangle is a commutative square in the
homotopy category of $\cat$. Therefore, to show that $\tau_{\le d}^{\mathcal{C}}\left(\underline{f_{A,B}}\right)$
is an equivalence, it is enough to show that $\underline{G'\left(F\left(f_{A,B}\right)\right)}$
is an equivalence. In fact, we shall show that $G'\left(F\left(f_{A,B}\right)\right)$
is an equivalence. Note that the composition of the left and then
bottom functors preserves binary products and since the right vertical
functor preserves products and is conservative, it follows that the
top functor $G'\circ F$ also preserves binary products. On the other
hand, $G'\circ F$ also preserves coproducts, since $F$ is left adjoint
(by the above discussion) and $G$ is an equivalence. Finally, in
$\alg_{\bb E_{\infty}}\left(\tau_{\le d}\mathcal{C}\right)$, the
canonical map from the coproduct to the product is an equivalence
by A.3.2.4.7.
\end{proof}

We now apply the above results to the study of reduced endomorphism operads.
For every unital $\infty$-operad $\mathcal{Q}$ and a symmetric monoidal
$\infty$-category $\mathcal{C}$, the symmetric monoidal $\infty$-category
$\alg_{\mathcal{Q}}\left(\mathcal{C}\right)$ is unital by \lemref{alg_unital}.
Hence, for every $X\in\alg_{\mathcal{Q}}\left(\mathcal{C}\right)$
we can consider the reduced endomorphism $\infty$-operad $\End_{\alg_{\mathcal{Q}}\left(\mathcal{C}\right)}^{\red}\left(X\right)$.
\begin{cor}
\label{cor:Reduced_Endomorphism_Truncatendness} Let $\mathcal{Q}$
be a reduced $n$-connected $\infty$-operad for some $n\ge-2$ and
let $\mathcal{C}$ be a $\left(d+1\right)$-topos with the Cartesian symmetric
monoidal structure for some $d\ge-2$. For every object $X\in\alg_{\mathcal{Q}}\left(\mathcal{C}\right)$,
the reduced endomorphism operad $\End_{\alg_{\mathcal{Q}}\left(\mathcal{C}\right)}^{\red}\left(X\right)$
is an essentially $\left(d-n-1\right)$-operad (ie all multi-mapping
spaces are $\left(d-n-2\right)$-truncated).
\end{cor}

\begin{proof}
The $\infty$-operad $\mathcal{E}=\End_{\alg_{\mathcal{Q}}\left(\mathcal{C}\right)}^{\red}\left(X\right)$
has a unique object, which we call $X$. We need to show that for every
$m\in\bb N$, the multi-mapping space $\mul_{\mathcal{E}}\left(X^{\left(m\right)},X\right)$
is $\left(d-n-2\right)$-truncated. By \lemref{Reduced_Endomorphism_Mapping_Space}
we have a fiber sequence
\[
\mul_{\mathcal{E}}(X^{\left(m\right)},X)\to\mul_{\alg_{\mathcal{Q}}\left(\mathcal{C}\right)}\left(X^{m},X\right)\to\map_{\alg_{\mathcal{Q}}\left(\mathcal{C}\right)}\left(X^{\sqcup m},X\right),
\]
where the fiber is taken over the fold map $\nabla\colon X^{\sqcup m}\to X$.
The fiber is equivalent to the space of lifts for the square
\[
\xymatrix{X^{\sqcup m}\ar[d]\ar[r]^{\nabla} & X\ar[d]\\
X^{m}\ar[r] & \term
.}
\]
Since $\mathcal{C}$ is an essentially $\left(d+1\right)$-category,
so is the Cartesian $\infty$-operad $\mathcal{C}_{\times}$ and, therefore, by \propref{Alg_d_Category}, so is $\alg_{\mathcal{Q}}\left(\mathcal{C}\right)$.
In particular, $X$ is $d$-truncated. Hence, by \propref{Lifts_Space_Truncatedness}, it is enough to show that the canonical map $X^{\sqcup m}\to X^{m}$ is $n$-connected. 
Since $\mathcal{Q}$ is $n$-connected, this follows from repeated application of \propref{Coproduct_Lemma}.
\end{proof}

\subsection{Topoi and the Reduced Endomorphism Operad}

In this subsection we describe a simple application of \corref{Reduced_Endomorphism_Truncatendness}.
Let $\mathcal{C}$ be an $\infty$-topos and let $\mathcal{C}_{*}$ be the
$\infty$-category of pointed objects in $\mathcal{C}$ with the Cartesian
symmetric monoidal structure. 
\begin{defn}
For a pair of integers $m,k\ge-2$, we denote by $\mathcal{C}_{*}^{\left[k,m\right]}\ss\mathcal{C}_{*}$
the full subcategory spanned by objects which are simultaneously $(k-1)$-connected (ie \emph{$k$-connective}) and $m$-truncated. 
\end{defn}

\begin{thm}
\label{thm:Application} Let $\mathcal{C}$ be an $\infty$-topos
and let $k,d\ge-2$. For every $X\in\mathcal{C}_{*}^{\left[k,2k+d\right]}$
the $\infty$-operad $\End_{\mathcal{C}_{*}}^{\red}\left(X\right)$
is an essentially $\left(d+1\right)$-operad. In particular, for $d=-1$,
the $\infty$-operad $\End_{\mathcal{C}_{*}}^{\red}\left(X\right)$
is either $\bb E_{0}$ or $\bb E_{\infty}$ and for $d=-2$, it is
$\bb E_{\infty}$.
\end{thm}

\begin{proof}
By A.5.2.6.10 and A.5.2.6.12, we have a commutative diagram of $\infty$-categories
\[
\xymatrix{\mathcal{C}_{*}^{\ge k}\ar[dr]_{\Omega^{k}}\ar[r]^{\sim} & \underline{\alg}_{\bb E_{k}}^{\grp}\left(\mathcal{C}_{*}\right)\ar[d]^{U}\\
 & \mathcal{C}_{*}
,}
\]
in which $U$ is the forgetful functor. Since the $k$-fold loop
space functor restricts to a functor $\mathcal{C}_{*}^{\left[k,2k+d\right]}\to\tau_{\le k+d}\mathcal{C}_{*}$,
we can restrict the above diagram to 
\[
\xymatrix{\mathcal{C}_{*}^{\left[k,2k+d\right]}\ar[dr]_{\Omega^{k}}\ar[r]^{\sim} & \underline{\alg}_{\bb E_{k}}^{\grp}\left(\tau_{\le k+d}\mathcal{C}_{*}\right)\ar[d]^{U}\\
 & \tau_{\le k+d}\mathcal{C}_{*}
.}
\]

The $\infty$-category $\underline{\alg}_{\bb E_{k}}^{\grp}\left(\tau_{\le k+d}\mathcal{C}_{*}\right)$ is a full subcategory of $\underline{\alg}_{\bb E_{k}}\left(\tau_{\le k+d}\mathcal{C}_{*}\right)$, which is itself equivalent to $\underline{\alg}_{\bb E_{k}}\left(\tau_{\le k+d}\mathcal{C}\right)$. 
The $\infty$-category $\tau_{\le k+d}\mathcal{C}$ is a $\left(k+d+1\right)$-topos (with the Cartesian symmetric monoidal structure) and $\bb E_{k}$
is $\left(k-2\right)$-connected. Thus, \corref{Reduced_Endomorphism_Truncatendness} implies that
for every $X$ in $\underline{\alg}_{\bb E_{k}}^{\grp}\left(\tau_{\le k+d}\mathcal{C}_{*}\right)$, the reduced endomorphism operad of $X$ is an essentially $\left(d+1\right)$-operad.

Let $d=-1$. We recall from \remref{Minus_One_Connected}
that if $\mathcal{P}\not\simeq\bb E_{0}$, then it is $\left(-1\right)$-connected. Therefore, if $\mathcal{P}$ is an essentially $0$-operad, then
$\mathcal{P}\simeq\bb E_{\infty}$. Hence, $\mathcal{P}$ is either
$\bb E_{0}$ or $\bb E_{\infty}$. 

Let $d=-2$. We get that $\mathcal{P}$ is an essentially $\left(-1\right)$-operad and hence equivalent to $\bb E_{\infty}$.
\end{proof}

For every reduced $\infty$-operad $\mathcal{P}$, the structure of
a $\mathcal{P}$-algebra on an object $X\in\mathcal{C}_{*}$ is equivalent
to the data of a map $\mathcal{P}\to\End_{\mathcal{C}_{*}}^{\red}\left(X\right)$.
Thus, if $X\in\mathcal{C}_{*}^{\left[k,2k-2\right]}$, then $X$ has
a unique $\mathcal{P}$-algebra structure for every reduced $\infty$-operad
$\mathcal{P}$. Combining this with the fact that for a pointed connected
object in an $\infty$-topos, a structure of an $\bb E_{\infty}$-algebra
is equivalent to an $\infty$-delooping, we get the following classical fact:
\begin{cor}
Let $\mathcal{C}$ be an $\infty$-topos and let $k\ge1$ be an integer.
Every $X\in\mathcal{C}_{*}^{\left[k,2k-2\right]}$ admits a unique
$\infty$-delooping. 
\end{cor}

In fact, we can get slightly more from \thmref{Application}. For
example,
\begin{cor}
Let $\mathcal{C}$ be an $\infty$-topos, let $k\ge1$ be an integer, and let $X\in\mathcal{C}_{*}^{\left[k,2k-1\right]}$. If $X$ admits an $H$-structure,
then it admits a unique $\infty$-delooping.
\end{cor}

\begin{proof}
By \thmref{Application}, the $\infty$-operad $\End_{\mathcal{C}_{*}}^{\red}\left(X\right)$
is either $\bb E_{0}$ or $\bb E_{\infty}$. On the other hand, the
existence of an $H$-structure is equivalent to $\End_{\mathcal{C}_{*}}^{\red}\left(X\right)\left(2\right)\ne\es$.
Thus, $X$ admits an $H$-structure if and only if $\End_{\mathcal{C}_{*}}^{\red}\left(X\right)\simeq\bb E_{\infty}$
if and only if $X$ admits a unique $\infty$-delooping.
\end{proof}

\subsection{The $\infty$-Categorical Eckmann\textendash Hilton Argument}

The main theorem of this paper is
\begin{thm}
\label{thm:Main_Theorem}For all integers $d_{1},d_{2}\ge-2$, given
a $d_{1}$-equivalence $\mathcal{P}\to\mathcal{Q}$ between two reduced
$\infty$-operads and a reduced $d_{2}$-connected $\infty$-operad
$\mathcal{R}$, the induced map $\mathcal{P}\otimes\mathcal{R}\to\mathcal{Q}\otimes\mathcal{R}$
is a $\left(d_{1}+d_{2}+2\right)$-equivalence.
\end{thm}

\begin{proof}
Set $d=d_{1}+d_{2}+2$. By \propref{d-equivalence}, it is enough
to show that for every $\left(d+1\right)$-topos \emph{$\mathcal{C}$}
with the Cartesian symmetric monoidal structure, the map 
\[
\map_{\op}\left(\mathcal{Q}\otimes\mathcal{R},\mathcal{C}\right)\to\map_{\op}\left(\mathcal{P}\otimes\mathcal{R},\mathcal{C}\right),
\]
induced by pre-composition with $f$, is a homotopy equivalence. Using
the tensor-hom adjunction, it is the same as showing that the map
\[
\map_{\op}\left(\mathcal{Q},\alg_{\mathcal{R}}\left(\mathcal{C}\right)\right)\to
\map_{\op}\left(\mathcal{P},\alg_{\mathcal{R}}\left(\mathcal{C}\right)\right)
\]
is an equivalence. The underlying category functor gives a commutative
diagram:
\[
\xymatrix{\map_{\op}\left(\mathcal{Q},\alg_{\mathcal{R}}\left(\mathcal{C}\right)\right)\ar[r]\ar[d] & \map_{\op}\left(\mathcal{P},\alg_{\mathcal{R}}\left(\mathcal{C}\right)\right)\ar[d]\\
\map_{\cat}(\underline{\mathcal{Q}},\underline{\alg}_{\mathcal{R}}\left(\mathcal{C}\right))\ar[r] & \map_{\cat}(\underline{\mathcal{P}},\underline{\alg}_{\mathcal{R}}\left(\mathcal{C}\right)).
}
\ \left(*\right)
\]

As $\underline{P}\to\underline{\mathcal{Q}}$ is an equivalence
of $\infty$-categories (both are equivalent to $\Delta^{0}$), the bottom map
is a homotopy equivalence. Hence, it suffices to show that the induced map
on the homotopy fibers is a homotopy equivalence for each choice of a base
point. 
A point in the space $\map\left(\Delta^{0},\alg_{\mathcal{R}}\left(\mathcal{C}\right)\right)$ is just an $\mathcal{R}$-algebra $X$ in $\mathcal{C}$. 
We denote by $\alg_{\mathcal{R}}\left(\mathcal{C}\right)_{X}$ the $\infty$-operad $\alg_{\mathcal{R}}\left(\mathcal{C}\right)$ pointed by $X$ viewed as an object of $\opunpt$. 
With this notation, we see that the homotopy fiber of the right vertical map is equivalent to 
\[
\map_{\opunpt}\left(\mathcal{P},\alg_{\mathcal{R}}\left(\mathcal{C}\right)_{X}\right).
\]

By \lemref{alg_unital}, the $\infty$-operad $\alg_{\mathcal{R}}\left(\mathcal{C}\right)$ is unital. Therefore, by the adjunction 
\[
\iota\colon \opred\adj\opunpt\colon \left(-\right)^{\red},
\]
the above mapping space is also equivalent to 
\[
\map_{\opred}\left(\mathcal{P},\End_{\alg_{\mathcal{R}}\left(\mathcal{C}\right)}^{\red}\left(X\right)\right)\simeq\map_{\op}\left(\mathcal{P},\End_{\alg_{\mathcal{R}}\left(\mathcal{C}\right)}^{\red}\left(X\right)\right),
\]
since $\opred\ss\op$ is a full subcategory. The induced map on the
fibers of the vertical maps in $\left(*\right)$ over $X$, is therefore
equivalent to 
\[
\map_{\op}\left(\mathcal{Q},\End_{\alg_{\mathcal{R}}\left(\mathcal{C}\right)}^{\red}\left(X\right)\right)\to\map_{\op}\left(\mathcal{P},\End_{\alg_{\mathcal{R}}\left(\mathcal{C}\right)}^{\red}\left(X\right)\right).
\]

Finally, since $\mathcal{\mathcal{C}}$ is a $\left(d+1\right)$-topos
and $\mathcal{R}$ is $d_{2}$-connected,  \corref{Reduced_Endomorphism_Truncatendness}
implies that the $\infty$-operad $\End_{\alg_{\mathcal{R}}\left(\mathcal{C}\right)}^{\red}\left(X\right)$
is an essentially $\left(d-d_{2}-1=d_{1}+1\right)$-operad. Since
$\mathcal{P}\to\mathcal{Q}$ is a $d_{1}$-equivalence, by \propref{d-homotopy_operad} the above map is a homotopy equivalence and this completes the proof. 
\end{proof}
\begin{example}
Let $\mathcal{P}\to\mathcal{Q}$ be a $d$-equivalence of reduced
$\infty$-operads. For every integer $k\ge0$, the induced map $\mathcal{P}\otimes\bb E_{k}\to\mathcal{Q}\otimes\bb E_{k}$
is a $\left(d+k\right)$-equivalence.
\end{example}

The $\infty$-categorical Eckmann\textendash Hilton argument is now an immediate consequence of \thmref{Main_Theorem}.
\begin{cor}
\label{cor:EHA}For all integers $d_{1},d_{2}\ge-2$, given two reduced
$\infty$-operads $\mathcal{P}$ and $\mathcal{Q}$, if $\mathcal{P}$
is $d_{1}$-connected and $\mathcal{Q}$ is $d_{2}$-connected, then
$\mathcal{P}\otimes\mathcal{Q}$ is $\left(d_{1}+d_{2}+2\right)$-connected.
\end{cor}

\begin{proof}
Since $\mathcal{P}$ is $d_{1}$-connected, the essentially unique map $\mathcal{P}\to\bb E_{\infty}$
is a $d_{1}$-equivalence. Hence, by \thmref{Main_Theorem}, the map
$\mathcal{P}\otimes\mathcal{Q}\to\mathcal{P}\otimes\bb E_{\infty}$
is a $\left(d_{1}+d_{2}+2\right)$-equivalence. Since $\bb E_{\infty}$
is also $d_{2}$-connected, by the same argument the induced map 
\[
\mathcal{P}\otimes\bb E_{\infty}\to\bb E_{\infty}\otimes\bb E_{\infty}\simeq\bb E_{\infty}
\]
is also a $\left(d_{1}+d_{2}+2\right)$-equivalence. The $\left(d_{1}+d_{2}+2\right)$-equivalences are closed under composition, and so the result follows (in fact, we know a posteriori that the map above is actually an equivalence of $\infty$-operads).
\end{proof}

We conclude this section (and this paper) with a couple of curious applications
of the $\infty$-categorical Eckmann\textendash Hilton argument. The first is
the classification of idempotent reduced $\infty$-operads. 
\begin{cor}
\label{cor:Idempotent_Operads}Let $\mathcal{P}$ be a reduced $\infty$-operad.
If $\mathcal{P}\otimes\mathcal{P}\simeq\mathcal{P}$, then $\mathcal{P}\simeq\bb E_{0}$
or $\mathcal{P}\simeq\bb E_{\infty}$.
\end{cor}

\begin{proof}
If $\mathcal{P}\not\simeq\bb E_{0}$, then, by \remref{Minus_One_Connected}, $\mathcal{P}$ is $d$-connected
for some $d\ge-1$. Therefore, by \thmref{Application}, $\mathcal{P}\otimes\mathcal{P}$ is $\left(2d+2\right)>d$ connected. Since $\mathcal{P}\simeq\mathcal{P}\otimes\mathcal{P}$, we can continue by induction and deduce that $\mathcal{P}$ is $\infty$-connected; hence $\mathcal{P}\simeq\bb E_{\infty}$.
\end{proof}

The second application is to a tensor product of a sequence of reduced $\infty$-operads. Given a sequence of reduced $\infty$-operads $\left(\mathcal{P}_{i}\right)_{i=1}^{\infty}$,
we can define the tensor product of them all $\bigotimes\limits _{i=1}^{\infty}\mathcal{P}_{i}$
as the colimit of the sequence
\[
\bb E_{0}\to\mathcal{P}_{1}\to\mathcal{P}_{1}\otimes\mathcal{P}_{2}\to\mathcal{P}_{1}\otimes\mathcal{P}_{2}\otimes\mathcal{P}_{3}\to\dots,
\]
where the $i$-th map is obtained by tensoring the essentially unique
map $\bb E_{0}\to\mathcal{P}_{i}$ with $\mathcal{P}_{1}\otimes\cdots\otimes\mathcal{P}_{i-1}$. 
\begin{example}
If we take $\mathcal{P}_{i}=\bb E_{1}$ for all $i$, then the additivity
theorem (A.5.1.2.2) implies that $\bigotimes\limits _{i=1}^{\infty}\bb E_{1}$
is the colimit of the sequence of $\infty$-operads
\[
\bb E_{0}\to\bb E_{1}\to\bb E_{2}\to\bb E_{3}\to\dots,
\]
which is $\bb E_{\infty}$. 
\end{example}

We offer the following generalization:
\begin{cor}
\label{cor:Infnite_Tensor}Let $\left(\mathcal{P}_{i}\right)_{i=1}^{\infty}$
be a sequence of reduced $\infty$-operads not equivalent to $\bb E_{0}$.
There is an equivalence of $\infty$-operads $\bigotimes\limits _{i=1}^{\infty}\mathcal{P}_{i}\simeq\bb E_{\infty}$.
\end{cor}

\begin{proof}
By \remref{Minus_One_Connected}, all $\mathcal{P}_{i}$-s are $\left(-1\right)$-connected.
By induction on $k$ and \corref{EHA}, the $\infty$-operad $\mathcal{P}_{1}\otimes\cdots\otimes\mathcal{P}_{k}$
is $\left(k-2\right)$-connected. For every $n\in\bb N$ we get 
\[
\left(\bigotimes\limits _{i=1}^{\infty}\mathcal{P}_{i}\right)\left(n\right)\simeq\colim\limits _{k}\left(\mathcal{P}_{1}\otimes\cdots\otimes\mathcal{P}_{k}\right)\left(n\right)\simeq\term
\]
and therefore $\bigotimes\limits _{i=1}^{\infty}\mathcal{P}_{i}\simeq\bb E_{\infty}$.
\end{proof}
For example, this implies that putting countably many compatible $H$-space
structures on a pointed connected space $X$ is the same as putting
an $\infty$-loop space structure on $X$.

\phantomsection\addcontentsline{toc}{section}{\refname}
\bibliography{Bib_EH}
\bibliographystyle{alpha}

\end{document}